\documentclass[11pt,twoside]{amsart}
\usepackage[all]{xy}
\usepackage {amssymb,latexsym,amsthm,amsmath,mathtools,comment,hyperref, longtable}
\usepackage{enumitem,color}

\usepackage{mathtools}
\DeclarePairedDelimiter\ceil{\lceil}{\rceil}
\DeclarePairedDelimiter\floor{\lfloor}{\rfloor}        

\long\def\symbolfootnote[#1]#2{\begingroup%
	\def\thefootnote{\fnsymbol{footnote}}\footnote[#1]{#2}\endgroup}
\newcommand{\Z}{\mathbb{Z}}

\newcommand{\diag}{\textup{diag}}
\newcommand{\GL}{\textup{GL}}

\newcommand{\rg}{\textup{rg}}
\newcommand{\rs}{\textup{rs}}
\newcommand{\s}{\textup{ss}}

\makeatletter
\def\imod#1{\allowbreak\mkern10mu({\operator@font mod}\,\,#1)}
\makeatother
\allowdisplaybreaks
\hypersetup{
	colorlinks=true,
	linkcolor=blue,
	filecolor=magenta,      
	urlcolor=cyan,
	citecolor=blue,
}     

\newtheorem{theorem}{Theorem}[section]
\newtheorem{lemma}[theorem]{Lemma}
\newtheorem{corollary}[theorem]{Corollary}
\newtheorem{proposition}[theorem]{Proposition}
\newtheorem*{theorem*}{Theorem}
\theoremstyle{definition}
\newtheorem{definition}[theorem]{Definition}

\newtheorem{example}[theorem]{Example}
\numberwithin{equation}{section}

\newcommand{\ignore}[1]{}

\newcommand{\mynote}[1]{}
\begin{document}
\setcounter{section}{0}
% document information
\title{Generating functions for the powers in $\text{GL}(n,q)$}
\author{Rijubrata Kundu}
\address{IISER Pune, Dr. Homi Bhabha Road, Pashan, Pune 411 008 India}
\email{rijubrata8@gmail.com}
\author{Anupam Singh}
\email{anupamk18@gmail.com}
\thanks{The first named author would like to acknowledge the support of NBHM PhD fellowship during 
this work. The second named author was supported by SERB core research grant CRG/2019/000271 during 
this work.}
%\date{}
\today

%%%%%%%%%%%%%%%%%%%%%%%%

\begin{abstract}
Consider the set of all powers $\text{GL}(n ,q)^M = \{x^M \mid x\in \text{GL}(n, q)\}$ for an 
integer $M\geq 2$. In this article, we aim to enumerate the regular, regular semisimple and 
semisimple elements as well as conjugacy classes in the set $\text{GL}(n, q)^M$, i.e., the elements 
or classes of these kinds which are $M^{th}$ powers. We get the generating functions for (i) regular 
and regular semisimple elements (and classes) when $(q,M)=1$, (ii) for semisimple elements and all 
elements (and classes) when $M$ is a prime power and $(q,M)=1$, and (iii) for all kinds when $M$ is 
a prime and $q$ is a power of $M$. 
\end{abstract}

\maketitle
%%%%%%%%%%%%%%%%%%%%%%%%

\section{Introduction}
One of the ways to approach problems in finite group theory is to study them statistically. This 
approach has been beautifully highlighted in the survey articles by some of the stalwarts of group 
theory, for example, see the articles by Shalev~\cite{sh1} and Dixon~\cite{di}. Here, we focus on 
the finite general linear group $\text{GL}(n, q)$ and aim to enumerate elements which are powers. 
The conjugacy classes of this group are given by the theory of Jordan and rational canonical forms. 
This theory, as well as the representation theory of this group (see~\cite{Gr}), uses partitions as 
a tool. Thus, representation-theoretic questions about this group naturally lead to several 
combinatorial questions. Fix an integer $M\geq 2$, and look at the power map $\omega_M \colon 
\text{GL}(n,q) \rightarrow \text{GL}(n, q)$ given by $x\mapsto x^M$. The central question here is to 
enumerate the image size $|\text{GL}(n,q)^M|$ and how many regular, regular semisimple, and 
semisimple elements it contains. That is, decide when a regular, regular semisimple or semisimple 
element is $M^{th}$ power in $\text{GL}(n, q)$. We would like to get the generating functions for 
the same. Since, if an element is in the image of $\omega$, all of its conjugates are so. We ask the 
enumeration questions about conjugacy classes as well. We always assume that $\omega_M$ is 
non-trivial, that is, $M$ is co-prime to the order of the group. We set the notation and explain 
the problems more precisely in Section~\ref{notation}. Let us briefly look at the results which 
motivated us to take up this project.

\subsection{Motivation}
The importance of looking at the power maps lies in one of the most active research areas 
of group theory at present, namely, the Waring-like problems for finite simple groups and, more 
generally, word maps (non-trivial) on groups. We refer an interested reader to the survey articles 
by Shalev~\cite{sh} and Larsen~\cite{La} on this subject and references therein. One of the key 
results due to Larsen,  Shalev and Tiep (Corollary 1.1.3~\cite{LST}) is that for a large enough 
non-Abelian finite simple group, every element is a product of two $M^{th}$ powers. Further, when 
$M$ is a product of two prime powers every element of every non-Abelian finite simple group is a 
product of two $M$ powers (see Theorem 1~\cite{GLOST}). In a recent work, the additive version of 
the Waring problem over finite fields is looked at by Kishore~\cite{Ki}. It is proved that for large 
enough $q$, every matrix over a finite field $\mathbb F_q$ is a sum of at most three $M^{th}$ 
powers. 

The surjectivity of power maps for algebraic groups has been studied by Chatterjee~\cite{Ch} and 
Steinberg~\cite{St}. Lubotzky (see~\cite{Lu}) showed that for a finite simple group a subset 
which is invariant under automorphisms of the group and contains the identity, can be obtained as 
an image of a word map. Estimates on the image size of a word map is obtained by Larsen and 
Shalev (see the Section 2~\cite{La}). The power map for the finite groups of Lie type 
$A_n$ and ${}^2A_n$ was studied in~\cite{gksv} which lead to certain interesting bounds on the size 
of $\text{GL}(n,q)^M$; further leading to the solution of one of Shalev's conjectures for power 
maps. An asymptotic formula for finite reductive groups when $q\rightarrow \infty$ for 
various quantities studied here, is obtained in~\cite{KKS} (see Theorem 1.1).  Thus, it is 
an interesting question to precisely determine the image of power maps for finite reductive groups. 
We begin this study with the group $\text{GL}(n,q)$ here.

Kung~\cite{ku} developed cycle index for $\text{GL}(n, q)$ analogous to Polya's cycle index theory 
for $S_n$. As an application, he obtained the generating function for the number of derangements. 
This 
was further applied by Stong~\cite{st1} to get several asymptotic results for various cycle 
statistics. Wall~\cite{wa, wa1} worked on the enumeration of conjugacy classes for classical 
groups. Fulman~\cite{Fu, Fu1} developed cycle index theory for other finite classical groups and 
provided a neater version of the cycle index generating function for $\text{GL}(n,q)$. He further 
took up a systematic study of enumeration and estimation of cyclic and separable and semisimple 
elements. Wall (see~\cite{wa2}) also studied the asymptotics of the proportion of cyclic 
and separable invertible matrices over $\mathbb F_q$. Our work generalizes some of these works as 
by taking $M=1$ we get the corresponding known formula. These works were followed up in~\cite{FNP} 
by Fulman, Neumann and 
Praeger where they extended the earlier results to all classical groups and obtained generating 
functions for regular, regular semisimple, semisimple conjugacy classes as well as elements. 
Britnell~\cite{Br1,Br2} studied this for special linear groups and unitary groups. Morrison has 
collected some of these generating functions for matrices in~\cite{Mo}. Estimates for 
powers in the symmetric group is due to Pouyanne (see~\cite{Po}) which is generalised to the wreath 
products $G\wr S_n$ in~\cite{KM}. These works motivated us to obtain generating functions for such 
elements which are powers. 

We also mention that our work provides a solution to an exercise by Stanley~\cite[Exercise 180, 
Chapter 1]{St2} (with difficulty rating 5, i.e, unsolved), which asks to count how many matrices 
over $\mathbb F_q$ have square roots. Finally, we mention that the discrete log problem which has 
implications in cryptography asks to find $M$ if we are given $x$ and $y=x^M$ in a group $G$. For 
the group $\text{GL}(n,q)$ this was studied by Menezes and Wu in~\cite{MW}. We hope our work will 
provide some insight into this subject too. 

\subsection{Organisation of the paper}
We divide this problem into two separate cases depending on, if $M$ and $q$ are coprime or not. The 
Jordan decomposition of elements necessitates this. For an element $g \in \text{GL}(n, q)$ we can 
write $g=g_s g_u=g_u g_s$ uniquely, where $g_s$ is semisimple part and $g_u$ is unipotent part of 
$g$. Thus, $g^M = g_s^Mg_u^M$. The semisimple elements are of order coprime to $q$, and the 
unipotent elements are of order a power of $q$. First, we take $M$ such that $(q,M)=1$. In this 
case, all unipotent elements will remain in the image of $\omega$. Hence, in the counting of 
$M^{th}$ powers, mainly, semisimple elements play a role. The generating function for regular and 
regular semisimple classes which are $M^{th}$ power, is in Theorem~\ref{Theorem-rs-rgCC}, and, for 
regular and regular semisimple elements it is in Theorem~\ref{Theorem-rs-rgAE}. In this case, to 
deal with semisimple elements and more general elements, we further assume $M=r^a$, where $r$ is a 
prime. We get the generating function for semisimple classes and semisimple elements which are 
$M^{th}$ powers, in Theorem~\ref{gen-fun-primepower} and for all elements in 
Theorem~\ref{Theorem-AE}. When $M$ is a prime we have a simpler formula for all these which we 
record in Section~\ref{generating-function-prime}. In Theorem~\ref{formula_M-powers_GL} we 
determine the exact value of the $M^{th}$ powers in some cases and show that limits obtained 
in~\cite{KKS} are achieved. For our work, we need to understand the factorisation of certain 
composed polynomials; thus we define M-power polynomials and study them in 
Section~\ref{M-power-polynomials}. In Section~\ref{Section-powers}, we develop combinatorial 
criteria for when a conjugacy class of $\text{GL}(n,q)$ is $M^{th}$ power. 
	
In the case when $(q,M)\neq 1$, the analysis could get complicated. We work with the case when $M$ 
is a prime, and $q$ is a power of $M$. In this case, all semisimple elements remain in the image. 
Miller~\cite{Mi} dealt with an instance of this when he counted squares in characteristic $2$. This 
part of our work generalises that of Miller. The generating function for conjugacy classes which are 
$M^{th}$ power is in Theorem~\ref{Theorem-modular}.

\subsection*{ Acknowledgement} The authors thank Amit Kulshrestha for his interest in this work. We 
also thank Will Sawin for his comments which helped us get a neat formula in 
Proposition~\ref{formula-nmqd}. The authors would like to express their gratitude to the 
anonymous referees whose comments helped improve the paper. 

%%%%%%%%%%%%%%%%%%%%%%%%%%%%%%%%%%%%
\section{Cycle index in $\text{GL}(n, q)$}\label{Scycle-index}

Conjugacy classes for the group $\text{GL}(n, q)$ is given by the theory of rational canonical 
forms. The enumeration of classes is done by Macdonald in~\cite{Ma} and Wall in~\cite{wa2}. To deal 
with several other enumeration problems, Kung~\cite{ku} and Stong~\cite{st1} developed the notion of 
cycle index and used it to get asymptotic results. Fulman~\cite{Fu, Fu1} gave alternate proofs for 
various generating functions, and further developed cycle index for other finite classical groups. 
We recall some of it, so that, we set notation for what follows in this article.  

The group $\text{GL}(n, q)$ is the set of invertible $n$-by-$n$ matrices over the finite field 
$\mathbb F_q$. Let $\Phi$ denote the set of all non-constant, monic, irreducible polynomials $f(x)$ 
(sometimes we simply write $f$) with coefficients in $\mathbb{F}_{q}$ which is not equal to the 
polynomial $x$. A conjugacy class of $\text{GL}(n, q)$ is determined by an associated combinatorial 
data as follows. Let $\Lambda$ be the set of all partitions $\lambda=(\lambda_1, \lambda_2, \ldots, 
\lambda_r)$ where $\lambda_1 \geq \lambda_2\geq  \ldots \geq \lambda_r \geq 0$ are integers. The set 
$\Lambda$ consists of $\lambda$ which are all possible partitions of all non-negative integers 
$|\lambda|$ where $|\lambda|$ is defined to be the sum of its parts. This includes the empty 
partition of $0$. To each $f\in \Phi$, we associate a partition $\lambda_f=(\lambda_{f,1}, 
\lambda_{f,2}, \ldots )$ of some non-negative integer $|\lambda_f|$. A conjugacy class of 
$\text{GL}(n, q)$ is in one-one correspondence with a function $\Phi \rightarrow \Lambda$ which 
takes value the empty partition on all but finitely many polynomials in $\Phi$, and satisfies 
$\displaystyle\sum_{f\in \Phi} |\lambda_f| deg(f) = n$. Thus, the  conjugacy class of an element 
$\alpha\in \text{GL}(n ,q)$ corresponds to the associated {\bf combinatorial data} 
$\Delta_{\alpha}$, which consists of distinct polynomials $f_1, \ldots, f_l$ and associated 
partitions $\lambda_{f_i}= (\lambda_{i_1}, \lambda_{i_2}, \ldots)$ for all $i$. In this notation, we 
keep only those $f_i$ on which the function for a conjugacy class takes value non-empty partitions 
$\lambda_{f_i}$. Given $\Delta_{\alpha}$, we can easily write down a representative of the conjugacy 
class of $\alpha$ as follows. For $f\in \Phi$, say, $f=x^d + a_{d-1}x^{d-1}+ \ldots + a_1 x + a_0$, 
we write the corresponding companion matrix $C(f) = \begin{psmallmatrix} 0& 0 & \cdots & 0 & -a_0\\ 
1&  0& \cdots & 0 & -a_1\\  & \ddots & \ddots &  \vdots& \vdots\\ & & 1& 0 & -a_{d-2}\\ &&&&\\ & & 
&1 & -a_{d-1} \end{psmallmatrix}$ where empty places represent $0$. Then, a matrix representative of 
the conjugacy class of $\alpha$ is the block-diagonal matrix $\diag(R_1, R_2,\ldots, R_l)$ where 
$R_i$ is the block diagonal matrix $\diag( J_{f_i, \lambda_{i_1}},  J_{f_i, \lambda_{i_2}}, \ldots 
)$ corresponding to various $f_i$. The matrix $J_{f_i, \lambda_{i_r}}$ is a block matrix of size 
$\lambda_{i_r}$ with each block size $deg(f_i)$ given as follows
$$ J_{f_i, \lambda_{i_r}} = \begin{psmallmatrix} C(f_i) & I & 0& \cdots &\cdots& 0\\
	&  C(f_i) & I &0 &\cdots& 0\\  &&\ddots&\ddots&\ddots&\vdots\\ & & &\ddots & \ddots & 0 \\
& &  && C(f_i) & I \\  &&  & & & C(f_i)
\end{psmallmatrix}$$
where $I$ denotes the identity matrix. 
Thus, the matrix $J_{f_i, \lambda_{i_r}}$ is a matrix of size $deg(f_i)\lambda_{i_r}$. We also 
remark that the conjugacy classes of $\text{GL}(n, q)$ correspond to the isomorphism class of 
$\mathbb F_q[x]$ module structure on the vector space $\mathbb F_q^n$.  

We recall some notation regarding partitions. The notation $\lambda \vdash n$ means $\lambda$ is a 
non-empty partition of $n$. The {\bf power notation} for partition $\lambda=1^{m_1(\lambda)}\cdots  
i^{m_i(\lambda)} \cdots n^{m_n(\lambda)}$ means that $i$ appears $m_i(\lambda)$ times in the 
partition where $m_i(\lambda) \geq 0$. The notation $\lambda^{'} = (\lambda_1', \lambda_2'\ldots )$ 
denotes the transpose partition of $\lambda$.  
	
With this knowledge of conjugacy class description the cycle index of $\text{GL}(n,q)$ is described 
(see~\cite{ku}) as follows. Let $x_{f,\lambda}$ be a variable associated to a pair $(f, \lambda)$ 
where $f$ is a polynomial and $\lambda$ a partition. The {\bf cycle index} is defined to be
$$ Z_{\text{GL}(n,q)} = \frac{1}{|\text{GL}(n,q)|} \sum_{\alpha \in \text{GL}(n, q)} 
\prod_{\begin{smallmatrix}f\in \Phi \\ |\lambda_f(\alpha)|>0 \end{smallmatrix}} 
x_{f,\lambda_f(\alpha)}. $$
The significance of this expression is that the coefficient of a monomial represents the probability 
that an element $\alpha$ of $\text{GL}(n, q)$ belongs to its conjugacy class, which is, one over the 
order of its centralizer (or that of any representative of its conjugacy class). 
Since the order of the centralizer of an element $\alpha$ in $\text{GL}(n, q)$ depends only on its 
combinatorial data $\Delta_{\alpha}$ and is given as follows: 
\begin{small}$$ \prod_{f\in \Delta_{\alpha}} \left( q^{\displaystyle deg(f).\sum_{i} 
(\lambda^{'}_{f,i})^2} \prod_{i\geq 1}\left (\frac{1}{q^{deg(f)}}\right )_{m_i(\lambda_f)} \right 
)$$
\end{small}
where the notation $\left(\frac{u}{q}\right)_i$ denotes $(1-\frac{u}{q})(1-\frac{u}{q^2})\cdots 
(1-\frac{u}{q^i})$.
Thus, the {\bf cycle index generating function} (see Section 2.1~\cite{Fu1}) is given as follows, 
\begin{small} 
\begin{equation}\label{cycle-index} 
\displaystyle 1+ \sum_{n=1}^{\infty} Z_{\text{GL}(n,q)} u^n=  \prod_{f\in \Phi} \left( 1+ 
\sum_{j \geq 1} \sum_{\lambda \vdash j} x_{f,\lambda} \frac{u^{j.deg(f)}}{\displaystyle q^{ 
deg(f).\sum_{i} (\lambda^{'}_{i})^2} \prod_{t \geq 1} 
\left(\frac{1}{q^{deg(f)}}\right)_{m_t(\lambda)}} \right).
\end{equation}
\end{small}
Fulman used this to get asymptotic behaviour of various quantities. We will make use of this 
equation frequently while writing various generating functions later.

%%%%%%%%%%%%%%%
\subsection{Regular, Semisimple and Regular Semisimple elements}

We recall some definitions here. An element $\alpha\in \text{GL}(n, q)$ is said to be {\bf regular} 
if the associated combinatorial data $\Delta_{\alpha}$ consists of polynomials $f_1, \ldots, f_l$ 
and each $\lambda_{f_i}$ has single part in its partition, i.e., partitions 
$\lambda_{f_1}=(\lambda_1), \ldots, \lambda_{f_l}=(\lambda_l)$ where $\lambda_i=|\lambda_{f_i}|$ for 
all $i$. A typical canonical element of this kind is $\diag(J_{f_1,\lambda_1}, J_{f_2,\lambda_2}, 
\ldots, J_{f_l,\lambda_l})$. An element $\alpha$ is said to be {\bf semisimple} if all parts in all 
partitions of $\Delta_{\alpha}$ has the value $1$, i.e., $\lambda_{f_i}=(1, 1, \ldots, 
1)=1^{|\lambda_{f_i}|}$ for all $i$. Thus, a typical canonical semisimple element would look like 
$\diag\left(\underbrace{C(f_1), \ldots, C(f_1)}_{|\lambda_{f_1}|}, \ldots, \underbrace{C(f_l), 
\ldots, C(f_l)}_{|\lambda_{f_l}|} \right)$. An element $\alpha$ is said to be {\bf regular 
semisimple} if the associated combinatorial data $\Delta_{\alpha}$ has all partitions 
$\lambda_{f_i}=(1)$ for all $i$. A typical canonical regular semisimple element would look like 
$\diag(C(f_1), C(f_2), \ldots, C(f_l))$. The set of regular, regular semisimple and semisimple 
elements in $\text{GL}(n, q)$ is denoted as $\text{GL}(n, q)_{\rg}$, $\text{GL}(n, q)_{\rs}$ and 
$\text{GL}(n ,q)_{\s}$ respectively. The number of all conjugacy classes, regular, regular 
semisimple and semisimple conjugacy classes in $\text{GL}(n, q)$ is denoted as $c(n,q)$, 
$c(n,q)_{\rg}$, $c(n,q)_{\rs}$ and $c(n,q)_{\s}$ respectively. The generating function for the 
number of conjugacy classes is (see~\cite[1.13]{Ma}) 
\begin{equation}\label{number-cclasses}
1+ \sum_{n=1}^{\infty} c(n,q) u^n = \prod_{i=1}^{\infty} \frac{1-u^i}{1-qu^i}.
\end{equation}

Let $\tilde N(q, d)$ be the number of monic, irreducible polynomials of degree $d$ over the field 
$\mathbb F_q$. We can express this using the M\"obius function $\mu(n)$ as follows:
$\displaystyle \tilde N(q,d)= \frac{1}{d} \sum_{r\mid d} \mu(r) q^{\frac{d}{r}}$.
We denote the number of monic, irreducible polynomials of degree $d$ over the field $\mathbb F_q$, 
except the polynomial $x$, by $N(q, d)$. Thus, we have $N(q, d)= \tilde N(q, d)$ except for $d=1$ 
and $N(q, 1) = \tilde N(q, 1) - 1 = q-1$. Other ways of computing this using generating functions 
can be found in~\cite[Lemma 1.3.10]{FNP}). Since, the regular (as well as semisimple) conjugacy 
classes in $\text{GL}(n, q)$ are in one-one correspondence with the monic polynomials of degree $n$ 
with non-zero constant term, we have
$$c(n,q)_{\rg}=c(n,q)_{\s}=q^n - q^{n-1}$$ 
and the generating function for $c(n,q)_{\rg}$ (and $c(n,q)_{\s}$) is 
\begin{equation}\label{number-rgclasses}
c(q,u)=1+\sum_{n=1}^{\infty} c(n,q)_{\rg} u^n = 1 +  \sum_{n=1}^{\infty} c(n,q)_{\s} u^n = 
\prod_{d\geq 1} (1-u^d)^{-N(q,d)}= \frac{1-u}{1-qu}.
\end{equation}
Now, we have the generating function for regular elements as follows:
$$ 1+ \sum_{n=1}^{\infty} \frac{|\text{GL}(n,q)_{\rg}|}{|\text{GL}(n,q)|} u^n = \prod_{d\geq 1} 
\left( 1 + \sum_{j=1}^{\infty} \frac{u^{jd}}{q^{(j-1)d}(q^d - 1)} \right)^{N(q,d)}.$$
The product on right hand side runs over degree $d$ of polynomials in $\Phi$ where we club together 
the ones of same degree. Wall~\cite[Equation 2.5]{wa2} has an alternate version of this:
$$1+ \sum_{n=1}^{\infty} \frac{|\text{GL}(n,q)_{\rg}|}{|\text{GL}(n,q)|}u^n = \prod_{d\geq 1} 
\left(1-\frac{u^d}{q^d} \right)^{-N(q,d)} \prod_{d\geq 1} \left( 1 + \frac{u^d}{q^d(q^d-1)} 
\right)^{N(q,d)}.$$
The generating function for semisimple elements is:
$$ 1+ \sum_{n=1}^{\infty} \frac{|\text{GL}(n,q)_{\s}|}{|\text{GL}(n,q)|} u^n = \prod_{d\geq 1} 
\left(1+\sum_{j=1}^{\infty} \frac{u^{jd}}{q^{\frac{j(j-1)}{2}d}\prod_{i=1}^n(q^{id}-1)} 
\right)^{N(q,d)}.$$

For a regular semisimple element, the characteristic polynomial is separable and it coincides with 
its minimal polynomial. Hence, the regular semisimple classes in $\text{GL}(n, q)$ are in one-one 
correspondence with the separable monic polynomials with non-zero constant term, thus we have 
(see~\cite[Theorem 1.1]{fjk} and \cite[Theorem 2.2]{FG})
$$ c(n,q)_{\rs}= \frac{q^{n+1} - q^n + (-1)^{n+1}(q-1)}{q+1}$$
and the generating function is 
\begin{equation}\label{number-rsclasses}
s(q,u)=1 + \sum_{n=1}^{\infty} c(n,q)_{\rs} u^n = \prod_{d\geq 1} 
(1+u^d)^{N(q,d)}=\frac{1-qu^2}{(1+u)(1-qu)}.
\end{equation}

For the number of regular semisimple elements (see~\cite[Equation 2.11]{wa2}) we have the following 
generating function,
$$1+ \sum_{n=1}^{\infty} \frac{|\text{GL}(n,q)_{\rs}|}{|\text{GL}(n,q)|}u^n= \prod_{d\geq 1} 
\left(1+\frac{u^d}{q^d-1} \right)^{N(q,d)}.$$

%%%%%%%%%%%%%%%%%%
\subsection{The main problem}\label{notation}

We fix an integer $M\geq 2$, and consider the power map $\omega_M \colon \text{GL}(n, q) \rightarrow 
\text{GL}(n,q)$ given by $x\mapsto x^M$. Let us denote the number of conjugacy classes, regular 
conjugacy classes, regular semisimple conjugacy classes and the semisimple conjugacy classes in the 
image of $\omega$ by $c(n,q, M)$, $c(n,q,M)_{\rg}$, $c(n,q,M)_{\rs}$, $c(n,q,M)_{\s}$ 
respectively. Our aim is to get the generating functions for these quantities. Further, let us 
denote the set of regular elements, regular semisimple elements and the semisimple elements in the 
image of $\omega$ by $\text{GL}(n, q)^M_{\rg}$, $\text{GL}(n, q)^M_{\rs}$ and $\text{GL}(n,q 
)^M_{\s}$ respectively.
We want to get the generating functions for the proportions of such elements $\displaystyle 
\frac{|\text{GL}(n,q)^M_{\rg}|}{|\text{GL}(n,q)|}$, $\displaystyle 
\frac{|\text{GL}(n,q)^M_{\rs}|}{|\text{GL}(n,q)|}$ and 
$\displaystyle\frac{|\text{GL}(n,q)^M_{\s}|}{|\text{GL}(n,q)|}$. To do this, first we need to 
determine the following: For a given $\alpha\in \text{GL}(n, q)$ where $\alpha$ is either regular, 
regular semisimple or semisimple, when does the equation
$$X^M=\alpha$$
have a solution in $\text{GL}(n, q)$? We need to characterise such $\alpha$ in terms of its 
combinatorial data $\Delta_{\alpha}$ which would lead to the required enumeration.

%%%%%%%%%%%%%%%%%%%%%%%
\section{$M$-power polynomials}\label{M-power-polynomials} 
	
In our work we need to deal with certain polynomials. We begin with describing them. Recall that the 
set of all monic, irreducible polynomials of degree $\geq 1$ over the field $\mathbb F_q$, except 
$x$, is denoted as $\Phi$. Counting this set is done using $N(q, d)$ described earlier. Let 
$M\geq 2$ be an integer. For a polynomial $f(x) = x^d+a_{d-1} x^{d-1} + \ldots + a_1x + a_0 \in 
\mathbb F_q[x]$ we denote the composed polynomial,
$$f(x^M) = x^{Md} + a_{d-1} x^{M(d-1)} + \ldots + a_1 x^M + a_0.$$ 
Now we define,
\begin{definition}[M-power polynomial]
A non-constant, irreducible, monic polynomial $f(x) \in \mathbb{F}_{q}[x]$ is said to be an {\bf 
M-power polynomial} if $f(x^M)$ has an irreducible factor of $deg(f(x))$. In general, a 
non-constant, monic polynomial $f$ is said to be an {\bf M-power polynomial} if each irreducible 
factor of $f$ is an M-power polynomial.
\end{definition} 
\noindent For example, the polynomial $x-a\in \mathbb F_q[x]$ is M-power if and only if $a\in 
\mathbb F_q^M$. We denote the set of monic, irreducible polynomials which are M-power by $\tilde 
\Phi^M$ and denote $\Phi^M = \tilde \Phi^M\setminus \{ x\}$. Let $\tilde N_M(q,d)$ be the number of 
polynomials of degree $d$ in $\tilde \Phi^M$ and $N_M(q,d)$ be that of $\Phi^M$. We have a simple 
relation $N_M(q,d)= \tilde N_M(q,d)$ except for $d=1$ and $N_M(q,1)= \tilde N_M(q,1)-1 
=\frac{q-1}{(M,q-1)}$. 
\begin{example}
Let us compute $N_2(q,2)$, i.e., the number of $2$-power polynomials of degree $2$ over $\mathbb 
F_q$. An irreducible polynomial $f$ of degree $2$ can be factored over $\mathbb F_{q^2}$ as 
$f(x)=(x-\alpha)(x-\sigma(\alpha))$ where $\sigma$ is the Frobenius automorphism. Now, 
$f(x^2)=(x^2-\alpha)(x^2-\sigma(\alpha))$ has a factor of degree $2$ over $\mathbb F_q$ if and only 
if $\alpha$ has a square root. Thus, for $N_2(q, 2)$ we need to count elements $\alpha$ which are in 
$(\mathbb F_{q^2}^*)^2$ but not in $\mathbb F_q$. We get, $N_2(q,2) = \frac{1}{2} 
\left(\frac{q^2-1}{(2, q^2-1)} -(q-1) \right)$.  
\end{example}
\noindent More generally we have the following,
\begin{proposition}\label{formula-nmqd}
For $d> 1$ we have,
$$N_M(q,d)= \frac{1}{d} \sum_{r\mid d} \mu(r)\frac{\left(M(q^{d/r}-1), (q^d-1) 
\right)}{(M,q^d-1)}.$$
\end{proposition}
\begin{proof}
Our proof is generalisation of the proof for $N(q, d)$ in~\cite{CM}.
Let $f$ be an irreducible M-power polynomial of degree $d> 1$.  That is, $f(x^M)$ has an 
irreducible factor of degree $d$. The irreducible polynomial $f$ is characterised with its root in 
$\mathbb F_{q^d}$. Consider the field extension $\mathbb F_{q^d}$ of $\mathbb F_q$ and the power map 
$\theta\colon \mathbb{F}_{q^d}^*\to \mathbb{F}_{q^d}^*$ defined by $\theta(x)=x^M$. Thus, M-power 
polynomial $f$ is characterised by an element in the image of $\theta$ which is primitive. Now, 
consider the set 
$$T= \{\alpha \in \mathbb{F}_{q^d}^* \mid \alpha\in \mathbb F_{q^d}^M, \alpha \notin \text{any 
proper subfield of } \mathbb{F}_{q^d}\}.$$ 
Then, we get $N_M(q,d)= \frac{1}{d}|T|$. Now we count the set $T$. To do this, we count $T(d,e)= 
\{ \alpha^M \in \mathbb{F}_{q^e} \mid  \alpha \in \mathbb{F}_{q^d} \}$ and apply the 
inclusion-exclusion principle. The number of pre-images of each element in $Im(\theta)$ is 
$|Ker(\theta)|$, which is $(M, q^d-1)$. Now, suppose $\alpha \in \mathbb{F}_{q^d}^{*}$, such that 
$\alpha^M \in \mathbb{F}_{q^e}$. Then $\left(\alpha^{M}\right)^{q^e}= \alpha^M$. Thus, the number of 
$\alpha \in \mathbb{F}_{q^d}^{*}$, which are solution of the equation $\alpha^{M(q^e-1)}=1$, is 
$(M(q^e-1), q^d-1)$. Hence we get $|T(d,e)|= \frac{(M(q^e-1), q^d-1)}{(M, q^d-1)}$. The result 
follows.
\end{proof}
\noindent For some small values of $M$ and $d$ we write $N_M(q,d)$ in tables below.

\begin{center}
\begin{tabular}{|c|c|c|c|}\hline
$q$ & $N_2(q, 2)$ & $N_2(q, 3)$ & $N_2(q, 4)$ \\ \hline
odd & $\frac{1}{4}(q-1)^2$ & $\frac{1}{6}(q^3-q)$ & $\frac{1}{8}(q^2-1)^2$ \\
even & $\frac{1}{2}(q^2-q)$ & $\frac{1}{3}(q^3-q)$ & $\frac{1}{4}(q^4-q^2)$ \\ \hline
\end{tabular}\vskip2mm
\end{center}

\begin{center}
\begin{tabular}{|c|c|c|c|}\hline
$q\imod 3$ & $N_3(q, 2)$ & $N_3(q, 3)$ & $N_3(q, 4)$ \\ \hline
$0$ & $\frac{1}{2}(q^2-q)$ & $\frac{1}{3}(q^3-q)$ & $\frac{1}{4}(q^4-q^2)$ \\
$1$ & $\frac{1}{6}(q^2-q)$ & $\frac{1}{9}(q^3-3q+2)$ & $\frac{1}{12}(q^4-q^2)$ \\
$2$ & $\frac{1}{6}(q-1)(q-2)$ &$\frac{1}{3}(q^3-q)$ & $\frac{1}{12}(q^4-q^2)$. \\ \hline
\end{tabular}\vskip2mm
\end{center}
The irreducibility of composed polynomials is studied in literature. We bring together the results 
when $M$ is a prime power and is co-prime to $q$.

%%%%%%%%%%%%%%
\subsection{$(q,M)=1$ and $M$ is a prime power}

In this subsection, we assume $(q,M)=1$. There is an extensive literature to determine the factors 
of polynomial $f(x^M)$ (more generally for composition of two polynomials) and their degrees.  For 
$(s, q)=1$, the notation $\mathfrak M(s; q)$ is the order of $q$ in $(\mathbb Z/s\mathbb 
Z)^{\times}$, that is $\mathfrak M(s; q)$ is the smallest integer such that $q^{\mathfrak M(s; q)} 
\equiv 1 \imod{s}$.  For an irreducible polynomial $f(x)$, which is not $x$, exponent of $f$ is the 
order of a root (which is same for all roots) of $f(x)$ in the multiplicative group $\bar {\mathbb 
F_q}^*$ (see~\cite[Chapter 3, Section 1]{LN}). We mention the following result due to Butler 
(see~\cite[Theorem in Section 3]{Bu}) which we need in sequel.
\begin{proposition}[Butler]\label{Butler}
Let $f(x)$ be an irreducible polynomial of degree $d$ over $\mathbb F_q$ and $(q,M)=1$. Let $t$ 
be the exponent of $f(x)$. Write $M=M_1 M_2$ in such a way that $(M_1,t)=1$ and each prime factor of 
$M_2$ is a divisor of $t$. Then,
\begin{enumerate}
\item $f(x^M)$ has no repeated roots.
\item The multiplicative order of each root of $f(x^M)$ in $\bar {\mathbb F_q}^*$ is 
$M_2tb$, for some $b$ that $b\mid M_1$.
\item Further, for a fixed $b\mid M_1$, the number of irreducible factors of $f(x^M)$ of 
which roots have the above order is $$\frac{M_2d\phi(b)}{\mathfrak M(M_2tb; q)}$$
and each of the factors is of degree $\mathfrak M(M_2tb; q)$, where $\phi$ is Euler's 
totient function.
\end{enumerate}
\end{proposition}
\noindent We use this to obtain some further information regarding M-power polynomials when $M$ is a 
prime power.
\begin{lemma}\label{M-is-a-power}
Let $M=r^a$ where $r$ is a prime and $(q,M)=1$. Suppose $f(x)$ is an irreducible polynomial of 
degree $d$ over $\mathbb{F}_{q}$ of exponent $t$.  Then we have the following:
\begin{enumerate}
\item If $r\nmid t$, the polynomial $f(x^M)$ has an irreducible factor of degree $d$, that 
is, $f$ is an M-power polynomial.
\item If $r\mid t$, the polynomial $f(x^M)$ factors as a product of $r^{a-i}$ irreducible 
polynomials each of degree $dr^{i}$ for some $1\leq i \leq a$.
\end{enumerate}
\end{lemma}
\begin{proof}
Let $\alpha$ be a root of $f(x)$ and $t$ be its multiplicative order. Then, $\mathbb 
F_q[\alpha]\cong \mathbb F_{q^d}$, hence $t\mid (q^d-1)$ (also gives $(t,q)=1$). In fact, because 
$\mathbb F_{q^d}$ is splitting field of $f$, the number $d$ is smallest with the property that 
$t\mid (q^d-1)$ (see~\cite[Theorem 3.3, 3.5]{LN}), hence $\mathfrak M(t;q)=d$.

First, let $r\nmid t$, then $M_1=M$ and $M_2=1$.   Thus, by taking $b=1$ in part 3 of 
Proposition~\ref{Butler}, $f(x^M)$ has an irreducible factor of degree $\mathfrak M(t;q)=d$. This 
shows that $f$ is an M-power polynomial.   

Now, let us consider the case when $r\mid t$, then $M_1=1$ and $M_2=M=r^a$. Once again applying 
Proposition~\ref{Butler}, all of the irreducible factors of $f(x^M)$ are of same degree, which is 
$\mathfrak M(r^at; q)=s$(say). That is $s$ is obtained from the equation $q^s\equiv 1 \imod{r^at}$. 
We claim that $d\mid s$ and $s\mid r^ad$ thus $s$ would have required form. Since $r^at\mid (q^s-1)$ 
hence $t\mid (q^s-1)$. This combined with the fact that the order of $q$ modulo $t$ is $d$, we get 
that $d\mid s$. Now, for the second one we show $q^{r^ad}\equiv 1 \imod{r^at}$ (which would give 
$s\mid r^ad$). We can write
$$(q^{r^ad}-1) = (q^{dr^{a-1}}-1)(q^{dr^{a-1}(r-1) }+  \cdots + q^d + 1).$$
Going modulo $r$ the second term on right becomes $0$ as $q^d\equiv 1 \imod r$ (as $r\mid t$). 
Thus this term is a multiple of $r$. By further reducing $a-1$, inductively, we get 
$(q^{r^ad}-1)=(q^d-1)r^a.h$ for some $h$. Notice that $t$ divides the first term. Hence the result. 
\end{proof}
\begin{corollary}
With notation as in the Lemma, let f(x) be an irreducible polynomial of degree $d$. Then, 
$\mathfrak{M}(r; q) \nmid d$ implies $f$ is an M-power polynomial. 
\end{corollary}
\begin{proof}
We claim that if $\mathfrak{M}(r; q) \nmid d$ then $r\nmid t$. Suppose $r\mid t$, then $r\mid 
(q^d-1)$. This gives $\mathfrak{M}(r;q) \mid d$, as $\mathfrak{M}(r;q)$ is the smallest with the 
property that $r\mid (q^{\mathfrak{M}(r;q)}-1)$. Now, the result follows by 
Lemma~\ref{M-is-a-power}.
\end{proof}
\noindent When $M$ is a prime we can get an easier way to decide if $f(x)$ is an M-power using 
$\mathfrak M(M; q)$ instead of the exponent which, in general, is difficult to compute.  
\begin{lemma}\label{M-is-prime}
Let $M$ be a prime and $(q,M)=1$. Let $f(x)$ be an irreducible polynomial of degree $d$ over 
$\mathbb{F}_{q}$. Then we have the following:
\begin{enumerate}
\item If $\mathfrak M(M; q)\nmid d$, then $f(x^M)$ factors as a product of an irreducible 
polynomial of degree $d$, and $\frac{d(M-1)}{lcm(\mathfrak M(M; q), d)}$ irreducible polynomials of 
degree $lcm(\mathfrak M(M;q), d)$. Thus, $f$ is an M-power polynomial.
\item If $\mathfrak M(M; q)\mid d$, then $f(x^M)$ is either irreducible or has a factor of 
degree $d$. Thus, if $f(x^M)$ is reducible it is M-power.
\end{enumerate}
\end{lemma}
\begin{proof}
Let us write $s=\mathfrak M(M; q)$. Let us begin with the case when $s \nmid d$. We must have 
$(M, t)=1$ and thus $M_1=M, M_2=1$. For if $(M, t)\neq 1$, i.e., $M\mid t$, combined with $t\mid 
(q^d-1)$ we get $M \mid (q^d-1)$. This gives, $s\mid d$ as $t$ is smallest with this property which 
is contrary to our assumption. Thus by Proposition~\ref{Butler}, $f(x^M)$ has  factors corresponding 
to $b=1$ and $b=M$. For the case $b=1$ we get a factor of degree $d$ as in the previous Lemma. It 
also has $\frac{d\phi(M)}{\mathfrak M(Mt; q)}$ factors of degree $\mathfrak M(Mt; q)$. We claim that 
$\mathfrak M(Mt; q)=lcm(\mathfrak M(M;q), d)$. But, this is clear because 
$\left(\mathbb{Z}/Mt\mathbb{Z}\right)^{\times} \cong \left(\mathbb{Z}/M\mathbb{Z}\right)^{\times} 
\times \left(\mathbb{Z}/t\mathbb{Z}\right)^{\times}$ because $(M, t)=1$. This completes the proof of 
first part.

Now, to prove the second part we have $s \mid d$. First we take $M\nmid t$. We have $M_1=M$ and 
$M_2=1$. Thus $f(x^M)$ has factors $\frac{d}{\mathfrak M(s; q)}=1$ irreducible polynomial of degree 
$\mathfrak M(s; q)=d$ and $\frac{d(M-1)}{\mathfrak M(tM; q)}$ irreducible polynomials each of degree 
$\mathfrak M(tM; q)=lcm(\mathfrak M(M; q), d)$.
Now take the case $M\mid t$. We have $M_1=1$ and $M_2=M$. Thus $f(x^M)$ is a product of 
$\frac{Md}{\mathfrak M(tM; q)}$ irreducible polynomials each of degree $\mathfrak M(tM; q)$ which is 
either $d$ or $Md$ (from second part of Lemma~\ref{M-is-a-power}). When $\mathfrak M(tM; q)=d$ we 
have $f$ an M-power, else $f(x^M)$ is irreducible.
\end{proof}

When $M=r^a$, we set some notation and do further counting of polynomials appearing in the 
Lemma~\ref{M-is-a-power} above. Denote $\widehat{N}(q,d) = N(q,d)-N_M(q,d)$. For $1\leq i\leq a$, we 
denote the number of irreducible polynomials $f(x)$ in $\Phi$ of degree $d$ with the property that 
all irreducible factors of $f(x^M)$ are of degree $dr^i$ by $N_M^{i}(q,d)$. Thus, from 
Lemma~\ref{M-is-a-power} it follows that,
$$N(q,d) = N_M(q, d) + \widehat{N}(q,d) = \sum_{i=0}^{a} N_M^{i}(q,d)$$
where, for notational convenience, we denote $N_M(q,d)$ as $N_M^{0}(q,d)$. We have the following 
formula for $N_M^{i}(q,d)$. 
\begin{proposition}\label{nmiqd}
Let $M=r^a$ where $r$ is a prime. For natural numbers $d$ and $e$, let $\widetilde{T}(d, e)$ 
denote the number of field generators of $\mathbb{F}_{q^e}$, that has a $M^{th}$ root in the field 
$\mathbb{F}_{q^d}$. Then, for $1\leq i\leq a$ we have,
$$N_M^{i}(q,d)= \frac{1}{d}\left(|\tilde{T}(dr^i, d)|-|\tilde{T}(dr^{i-1}, d)|\right).$$
\end{proposition}
\begin{proof}
The proof is similar to that of Proposition~\ref{formula-nmqd}. Since $\widetilde{T}(d,e)$ 
denotes the number of field generators of $\mathbb{F}_{q^e}$, that has a $M^{th}$ root in the field 
$\mathbb{F}_{q^d}$ we have, 
$$|\tilde{T}(d,e)|= \sum_{r\mid e} \mu(r)\frac{\left(M(q^{e/r}-1), q^d-1\right)}{(M,q^d-1)}$$
where $\mu$ is the Mobius function. Comparing with the proof of Proposition~\ref{formula-nmqd}, 
we note that $\frac{1}{d}|\tilde{T}(d,d)|=N_M(q,d)$.

To compute $N_M^{i}(q,d)$, we need to find the number of field generators of $\mathbb{F}_{q^d}$ 
which possess $M^{th}$ root in the field $\mathbb{F}_{q^{dr^i}}$ but not in any smaller subfield 
between $\mathbb{F}_{q^{dr^i}}$ and $\mathbb{F}_{q^d}$. The set $\widetilde{T}(dr^{i}, d)$ gives the 
total number of field generators of $\mathbb{F}_{q^d}$, which posses $M^{th}$ root in the field 
$\mathbb{F}_{q^{dr^i}}$. Thus, to get $N_M^{i}(q,d)$ we need to subtract $|\widetilde{T}(dr^{i-1}, 
d)|$. Finally we also note that we $d$ such elements correspond to a polynomial thus we divide by 
that to get the result.
\end{proof}
We look at an example here. 
\begin{example}
Let us take $M=2^2$ and $d=1$. We have already seen $N_4^0(q, 1)=N_4(q, 1)=\frac{q-1}{(4,q-1)}$. 
Now, $N_4^1(q, 1)$ counts the number of polynomials $x-\lambda$, such that $x^4-\lambda$ factors as 
a product of two irreducible degree $2$ polynomials. Thus we have, $N_4^1(q, 1)=\left(\frac{(q-1)(4, 
q+1)}{(4, q^2-1)}-\frac{q-1}{(4, q^2-1)}\right) = \frac{q-1}{(4, q^2-1)}\left((4,q+1)-1\right)$. 
Finally, $N_4^2(q, 1)$ counts the number of polynomials $x-\lambda$ such that $x^4-\lambda$ is 
irreducible. Thus,
\begin{eqnarray*}
N_4^2(q, 1)&=& 
\frac{(4(q^2-1),q^4-1)}{(4,q^4-1)}-\frac{(4(q-1),q^4-1)}{(4,q^4-1)}-\frac{(4(q-1),q^2-1)}{(4,q^2-1)} 
\\
&=&(q-1)\left(\frac{(q+1)(4,q^2+1)}{(4,q^4-1)}-\frac{(4,q^3+q^2+q+1)}{(4,q^4-1)}-\frac{(4,q+1)}{(4,
q^2-1)}\right).
\end{eqnarray*}
\end{example}
Before moving any further, we set the following notation when $M=r^a$. For $1\leq i\leq a$, denote 
the set of all polynomials $f\in \Phi$ such that $f(x^M)$ has $r^{a-i}$ irreducible factors each of 
degree $r^i deg(f)$, by $\Phi_{M,i}$. Then by Lemma~\ref{M-is-a-power} we have $$\Phi = \Phi^M 
\bigcup \bigcup_{i=1}^a \Phi_{M, i} = \bigcup_{i=0}^a \Phi_{M, i} $$
where, for convenience, we denote $\Phi^M=\Phi_{M ,0}$. 
	
%%%%%%%%%%%%%%%%%%%%%%%%%%%%%%%%%%%%%
	
\subsection{$(M,q)=1$ and $M$ a prime} 

When $M$ is a prime, $N_M(q,d)$ can be obtained using $N(q,d)$ which we explore here. Observe that 
$1\leq \mathfrak{M}(M,q)\leq M-1$, and thus $(\mathfrak{M}(M,q), M)=1$. 
	
\begin{proposition}\label{reln-nmqd-nqd}
Let $M\geq 2$ be a prime and $(M, q)=1$. Suppose, $\mathfrak{M}(M,q) \mid d$. Then,
	$$\displaystyle N_M(q,d)=N(q,d)-\frac{M-1}{tM^{k+1}}N(q^{M^kt},y)$$
where $t= \mathfrak{M}(M,q)$ and $d$ is written as $d=t. M^k.y$ with $k\geq 0$ and $M\nmid y$.
\end{proposition}
\begin{proof} Let us write $t= \mathfrak{M}(M,q)$. For $M$ a prime in Proposition~\ref{formula-nmqd}, we have,
\begin{eqnarray*}
N_M(q,d) &=& \frac{1}{d(M,q^d-1)} \sum_{r\mid d} \mu(r) \left(M(q^{d/r}-1), (q^d-1) \right)\\
&=& \frac{1}{Md}\sum_{r\mid d} \mu(r)(q^{d/r}-1)\left(M,\frac{q^d-1}{q^{d/r}-1}\right)
\end{eqnarray*}
where we used $(M, q^d-1)=M$ since $t\mid d$. Therefore,
\begin{eqnarray*}
N(q,d) - N_M(q,d) & =& \left(\frac{1}{d}\sum_{r\mid 
d}\mu(r)(q^{d/r}-1)\right)-\left(\frac{1}{Md}\sum_{r\mid d} \mu(r)(q^{d/r}-1)\left(M, 
\frac{q^d-1}{q^{d/r}-1}\right)\right)\\ 
& =& \frac{1}{d}\sum_{r\mid d} 
\left(\mu(r)(q^{d/r}-1)\left[1-\frac{\left(M,\frac{q^d-1}{q^{d/r}-1}\right)}{M}\right]\right).
\end{eqnarray*}
Now, we claim that the sum above can be simplified as follows:
\begin{eqnarray*}
\sum_{r\mid d} 
\left(\mu(r)(q^{d/r}-1)\left[1-\frac{\left(M,\frac{q^d-1}{q^{d/r}-1}\right)}{M}\right]\right)=\frac{
M-1 } { M } \sum_{r\mid y}\mu(r)(q^{d/r}-1)
\end{eqnarray*}
where, $d=M^k.t.y$, with $(M,y)=1$. Using this claim, we can easily complete the proof as follows:
\begin{eqnarray*}
&& N(q,d)-N_M(q,d )= \frac{M-1}{Md}\sum_{r\mid y}\mu(r)(q^{d/r}-1)\\
&=&\left(\frac{M-1}{M^{k+1}.t}\right)\left(\frac{1}{y}\sum_{r\mid 
y}\mu(r)((q^{M^kt})^{y/r}-1)\right)= \frac{M-1}{M^{k+1}.t}N(q^{M^kt},y).
\end{eqnarray*}

Now, to prove the claim, we make several cases for $r\mid d$ depending on $d = M^kty$ where 
$(M,y)=1$. We begin with the 
following observation: $\frac{q^d-1}{q^{d/r}-1}= q^{\frac{d}{r}(r-1)} + \cdots +q^{\frac{d}{r}}+1$, 
and if $t\mid \frac{d}{r}$ then using the definition of $t$, we get $ \frac{q^d-1}{q^{d/r}-1}= 
q^{\frac{d}{r}(r-1)} + \cdots +q^{\frac{d}{r}}+1 \equiv r \imod M$.

\textbf{Case 1:} Let $r\mid y$ then $t\mid \frac{d}{r}$ and $(r, M)=1$. From the calculation 
above $\frac{q^d-1}{q^{d/r}-1} \equiv r \imod M$, and we get $\left(M,\frac{q^d-1}{q^{d/r}-1}\right) 
= 1$. Hence, the term in the claim is $1-\frac{(M,\frac{q^d-1}{q^{d/r}-1})}{M}=1-\frac{1}{M}$.

\textbf{Case 2:} Suppose now that $r\mid d$ but $r\nmid y$. In this case we show that  
$$\mu(r)(q^{d/r}-1)\left[1-\frac{(M,\frac{q^d-1}{q^{d/r}-1})}{M}\right]=0.$$
At first, let $(r,t)=1$ but $r\mid M^ky$, say $r=M^as$ where $(s,M)=1$ since $(M,y)=1$. If $a\geq 2$ we have $\mu(r)=0$ and we are done. Now, 
suppose $r=Ms$ then $t\mid \frac{d}{Ms}$ and $\frac{q^d-1}{q^{d/r}-1}= \frac{q^d-1}{q^{d/M}-1} = Ms 
\equiv 0 \imod M$. Hence, $\left(M, \frac{q^d-1}{q^{d/r}-1}\right) = M$ by the observation made 
above, and we get the required result.

Now suppose, $r\mid Mty$, where $(t,M)=(y,M)=1$ with $(r,t)\neq 1$. If $(M,\frac{q^d-1}{q^{d/r}-1})=M$ we are done once again. Now, suppose that 
$(M,\frac{q^d-1}{q^{d/r}-1})=1$. In this case it is clear that $t\mid \frac{d}{r}$. Let $(r,t)=a\neq 1$. We can also assume in this case that $M\nmid r$. Otherwise, $(M,\frac{q^d-1}{q^{d/r}-1})=M$ by the observation made before case 1. This gives a contradiction. Thus, we have $r\mid t.y$ with $(r,t)=a\neq 1$. Since $t\mid \frac{d}{r}$ we must have $a\mid y$. Thus, we must have $r\mid y$ in this case which once again gives a contradiction. This completes the proof of our claim.
\end{proof} 
\noindent We list the complete result below. 
\begin{corollary}\label{final-reln-nmqd-nqd}
Suppose $M\geq 2$ is a prime with $(M,q)=1$. Let $\mathfrak{M}(M,q)=t$. Then,
\[ N_M(q,d)=
\begin{cases}
N(q,d) & \text{when }t\nmid d\\
N(q,d)-\frac{M-1}{M^{k+1}t}N(q^{M^kt},y)  & \text{when } d=M^k.t.y\text{ for some }k\geq 0 
\text{, and M $\nmid$ y }
\end{cases}
\]
\end{corollary}
\begin{proof}
Follows from the Proposition above and Lemma~\ref{M-is-prime}.
\end{proof}

\begin{example}
Let $M=2$. Then $t=1$ is the only choice. We have,
	\begin{center}
$\displaystyle N_2(q,5)=N(q,5)-\frac{1}{2}N(q,5)=\frac{1}{10}(q^5-q)$
\end{center}
since $k=0$ in the above case. Similarly,
\begin{center}
$\displaystyle 
N_2(q,6)=N(q,6)-\frac{1}{4}N(q^2,3)=N(q,6)-\frac{1}{12}(q^6-q^2)=\frac{1}{12}(q^6-q^2-2q^3+2q).$
\end{center}
since $k=1$ in the above case, and $\displaystyle N(q,6)=\frac{1}{6}(q^6-q^2-q^3+q)$.
\end{example}

\begin{example}
Let $M=3$. The possible values of $t$ are 1 and 2. Let us first take $t=1$, that is, $q\equiv 
1(\text{mod }3)$. We calculate $N_3(q,5)$.
	$$N_3(q,5)=N(q,5)-\frac{1}{3}N(q,5)=\frac{2}{15}(q^5-q).$$
$$N_3(q,6)=N(q,6)-\frac{1}{9}N(q^3,2)=N(q,6)-\frac{1}{18}(q^6-q^2)=\frac{1}{18}
(2q^6-3q^3-2q^2+3q).$$
since $k=1$ in this case.

\noindent Let us now take $t=2$, that is, $q\equiv 2(\text{mod }3)$. In this case when $d=5$, 
$t\nmid d$. Thus,
$$N_3(q,5)=N(q,5)=\frac{1}{5}(q^5-q).$$
When $d=6$, we have $k=1$. Thus,
$$N_3(q,6)=N(q,6)-\frac{1}{18}N(q^6,1)=N(q,6)-\frac{1}{18}(q^6-1)=\frac{1}{18}
(2q^6-3q^3-3q^2+3q+1)$$
\end{example}

%%%%%%%%%%%%%%
\section{$M^{th}$ powers in $\text{GL}(n,q)$}\label{Section-powers}
	
We consider the power map $\omega\colon \text{GL}(n,q) \rightarrow \text{GL}(n,q)$ given by 
$x\mapsto x^M$. We further assume that $(q,M)=1$. Let $\alpha\in \text{GL}(n,q)$ with combinatorial 
data $\Delta_{\alpha}$ which determines $\alpha$ up to conjugacy. We have introduced this in 
Section~\ref{Scycle-index}. Conversely, to such a data we have associated representative matrix of 
the conjugacy class which we make use of in the sequel for the further computations. 
\begin{lemma}\label{jordan-block-power}
Let $\gamma \in \mathbb F_q^*$ and $J_{\gamma,n}= \begin{psmallmatrix} \gamma & 1 & 0 & 0 
&\cdots \\ &\gamma &1 &0 &\cdots \\ &&\ddots&\ddots&\vdots \\ &&&\gamma&1\\ &&&&\gamma 
\end{psmallmatrix}$ be the Jordan matrix of size $n$. Then, $(J_{\gamma, n})^{M}$ is conjugate to 
$J_{\gamma^{M}, n}$. 
\end{lemma} 
\begin{proof}
We write $J_{\gamma, n}= \gamma I_{n} + N$ and notice that $N$ is a nilpotent matrix satisfying 
$N^n=0$ and $N^{k}\neq 0$ for all $k<n$. Thus, 
$$(J_{\gamma, n})^M= (\gamma I_n + N)^M= \gamma^M I_{n} + \binom{M}{1} \gamma^{M-1}I_{n}N + 
\cdots + N^{M}.$$ 
Hence, $(J_{\gamma, n})^M$ has all diagonal entries $\gamma^M$, and all entries above the 
diagonal $M\gamma^{M-1}$. Since $(q,M)=1$, the result follows.
\end{proof}
Let $f\in \Phi$ be a polynomial of degree $k\geq 1$. Then, $f$ splits over $\mathbb F_{q^k}$. The 
Galois automorphisms of this field is obtained by taking powers of the Frobenius automorphism 
denoted as $\sigma_k$. Let $f_M\in \mathbb F_q[x]$ be the minimal polynomial of $M^{th}$ power of 
one of the roots of $f$. If $\eta$ is a root of $f$ then other roots of $f$ are $\sigma_k^i(\eta)$ 
for $0\leq i \leq k-1$, and $f_M$ is the minimal polynomial of $\eta^k$. Note that $f_M$ is uniquely 
associated to $f$, say it is of degree $d$. Then, $d\mid k$ and $\mathbb F_{q^d}$ is the splitting 
field of $f_M$ which is a subfield of $\mathbb F_{q^k}$. We have the following, 
\begin{lemma}\label{power-poly}
Let $f\in \Phi$ be of degree $k\geq 1$, and $f_M$ be minimal polynomial of $M^{th}$ power of a 
root of $f$ of degree $d$. Then, exactly $\frac{k}{d}$ roots of $f(x)$ raised to the power $M$ give 
a root of $f_M(x)$. Further, $f(x)$ is an irreducible factor of the polynomial $f_M(x^M)$.
\end{lemma}
\begin{proof}
Let $\eta\in \mathbb F_{q^k}$ be a root of $f(x)$, and $\eta^M=\zeta \in \mathbb F_{q^d}$  with 
minimal polynomial $f_M(x)$. Then, the set of roots of $f$ is $\mathcal S= \{\sigma_k^i(\eta)\mid 
0\leq i \leq k-1\}\subset \mathbb F_{q^k}$ and under the $M^{th}$ power map this goes inside the set 
$\tilde{ \mathcal S}= \{\sigma_d^i(\zeta)\mid 0\leq i \leq d-1\}\subset \mathbb F_{q^d}$ which are 
roots of $f_M$. Thus the first statement follows. 

We note that $\eta$ is a root of $f_M(x^M)$ as $(x^M-\zeta)=(x^M-\eta^M)$ is a factor. Thus, the 
minimal polynomial of $\eta$, which is $f$, divides $f_M(x^M)$. 
\end{proof}
\noindent Now, we consider slightly more general case of $\alpha \in \text{GL}(n,q)$ with 
combinatorial data $\Delta_{\alpha}$ and determine the data $\Delta_{\alpha^M}$ for $\alpha^M$. 
\begin{proposition}\label{M-power-partition-split} 
Suppose $\alpha\in \text{GL}(n,q)$ with $\Delta_{\alpha}$ consisting of a single irreducible 
polynomial $f$ of degree $k$ and $\lambda_{f}=(\lambda_1, \ldots ,\lambda_l)$ where 
$|\lambda_{f}|=\frac{n}{k}$. Let $f_M$ be the minimal polynomial of a $M^{th}$ power of a root of 
$f$, say of degree $d$. Then, $\Delta_{\alpha^M}$ consists of a single polynomial $f_M$ and 
$|\lambda_{f_M}|=\frac{n}{d}$ with 
$$ \lambda_{f_M}= (\underbrace{\lambda_1, \ldots , \lambda_1}_s , \ldots, \underbrace{\lambda_l, 
\ldots, \lambda_l}_s)$$
where $s=\frac{k}{d}$.
\end{proposition}
\begin{proof}
Recall from Section~\ref{cycle-index}, the associated representative of conjugacy class 
corresponding to $\alpha$ is the matrix $A_{\alpha}=\diag(J_{f, \lambda_{1}}, \ldots, J_{f, 
\lambda_{l}})$ where $J_{f, \lambda_{i}}$ is a block matrix of size $\lambda_i.k$ with each block 
size $k$ and diagonals $C(f)$. Thus to compute $A_{\alpha}^M$  we first look at a single block 
$J_{f, \lambda_{i}}$. Note that $J_{f, \lambda_{i}} = D+N$ where $D=\diag(C(f),\ldots, C(f))$ and 
$N=\begin{psmallmatrix} 0& I &&&&\\ &0 &I && &\\ &&&\ddots&\ddots& \\ &&&&0&I\\ &&&&&0 
\end{psmallmatrix}$ and $DN=ND$. Thus,
$$J_{f, \lambda_{i}}^M = (D+N)^M = D^M + MD^{M-1}N + \cdots $$
where $D^M= \diag(C(f)^M, \ldots, C(f)^M) $. Since $(q,M)=1$, the data $\Delta_{\alpha^M}$ will 
depend only on how $C(f)^M$ splits up.
Over $\mathbb F_{q^k}$, the companion matrix $C(f)$ is conjugate to the diagonal matrix 
$\diag(\eta, \sigma(\eta), \ldots, \sigma^{k-1}(\eta))$ where $Gal(\mathbb F_{q^k}/\mathbb 
F_q)=<\sigma>$. Thus, $J_{f, \lambda_{i}}$ is conjugate to the matrix $\diag(J_{\eta, \lambda_i}, 
J_{\sigma(\eta), \lambda_i}, \ldots, J_{\sigma^{k-1}(\eta), \lambda_i})$. Now, using 
Lemma~\ref{jordan-block-power}, we get $J_{f, \lambda_{i}}^M$ is conjugate to $\diag(J_{\eta^M, 
\lambda_i}, J_{\sigma(\eta)^M, \lambda_i}, \ldots, J_{\sigma^{k-1}(\eta)^M, \lambda_i})$ over 
$\mathbb F_{q^k}$.   
Thus, by grouping together the blocks where the conjugates of $\zeta=\eta^M$ appear, we get 
$J_{f, \lambda_{i}}^M$ is conjugate to 
$$\diag\left( \underbrace{J_{\zeta, \lambda_i}, J_{\sigma_d(\zeta), \lambda_i}, \ldots, 
J_{\sigma_d^{d-1}(\zeta), \lambda_i}}_{1}, \ldots, \underbrace{J_{\zeta, \lambda_i}, 
J_{\sigma_d(\zeta), \lambda_i}, \ldots, J_{\sigma_d^{d-1}(\zeta), \lambda_i}}_{s} \right)$$
with $s$ many grouped blocks, because of Lemma~\ref{power-poly}. Further, notice that 
$$\diag(J_{\zeta, \lambda_i}, J_{\sigma_d(\zeta), \lambda_i}, \ldots, J_{\sigma_d^{d-1}(\zeta), 
\lambda_i})$$ where $Gal(\mathbb F_{q^d}/\mathbb F_q)=<\sigma_d>$, is conjugate to $J_{f_M, 
\lambda_i}$; which is a block matrix of size $\lambda_i.d$ with block size $d$.  Thus, $J_{f, 
\lambda_{i}}^M$ is conjugate to $\diag(\underbrace{J_{f_M,\lambda_i}, \ldots, 
J_{f_M,\lambda_i}}_s)$. This gives us the required result.
\end{proof}
\noindent In this proposition, the partition $\lambda_{f_M}$ can be easily visualized in the power 
notation of partitions where multiplicity of each part gets multiplied by $s$. 
We can generalise the above result to more general set up where $\Delta_\alpha$ has more than one 
polynomials but the minimal polynomial of $M^{th}$ power of a root of each one of them is a single 
polynomial. 
\begin{proposition}\label{M-power-single-poly}
Suppose $\alpha\in \text{GL}(n,q)$ with associated data $\Delta_{\alpha}$ consisting of 
polynomials $f_i \in \Phi $ of degree $d_i$ and partitions $\lambda_{f_i}=(\lambda_{i_1}, 
\lambda_{i_2}, \ldots)$, $1\leq i \leq l$. Let $h(x)$ be a polynomial of degree $d$ which is the 
minimal polynomial of $M^{th}$ power of a root of each $f_i$ for all $i$ (that is, $(f_i)_M=h, 
\forall i$). Then, $\Delta_{\alpha^M}$ consists of the single polynomial $h(x)$ and partition
$$\lambda_{h(x)}=\left(\underbrace{\lambda_{1_1}^{\frac{d_1}{d}}, \lambda_{1_2}^{\frac{d_1}{d}}, 
\ldots}, \ldots, \underbrace{\lambda_{i_1}^{\frac{d_i}{d}}, \lambda_{i_2}^{\frac{d_i}{d}}, \ldots}, 
\ldots, \underbrace{\lambda_{l_1}^{\frac{d_l}{d}}, \lambda_{l_2}^{\frac{d_l}{d}}, \ldots} \right)$$	
with $|\lambda_{h(x)}|=\frac{n}{d}$. 
\end{proposition}
\noindent The proof of this follows from the earlier proposition. A more general version of this 
Proposition can be written where, we have $h_1(x), \ldots, h_m(x)$ which are the minimal polynomials 
of $M^{th}$ powers of a subset of $f_i$'s. We also note that in this proposition the partition 
obtained need not be ordered. However, there is no loss here if we make it ordered.

Now, we apply the results obtained so far to get certain classes which are $M^{th}$ power. 
\begin{proposition}\label{M-power-determined}
Let $\alpha \in \text{GL}(n,q)$ with combinatorial data $\Delta_{\alpha}$. Suppose, each 
partition $\lambda_{f_i}$ in $\Delta_{\alpha}$ has all its parts distinct. Then, $X^M=\alpha$ has a 
solution in $\text{GL}(n,q)$ if and only if $f_i$ is M-power for all $i$ (that is, $f_i(x^M)$ has an 
irreducible factor of degree $deg(f_i)$ for all $i$).
\end{proposition}
\begin{proof}
It suffices to prove this for a single polynomial $i=1$ case. Thus we may assume, 
$\alpha \in \text{GL}(n,q)$ with $\Delta_{\alpha}$ consisting of a single polynomial $h(x) \in 
\Phi$ of degree $d$ and partition $\lambda_{h(x)}$ of $\frac{n}{d}$ has all of its parts distinct. 
Now, we need to prove $X^M=\alpha$ has a solution in $\text{GL}(n,q)$ if and only if the polynomial 
$h(x^M)$ has an irreducible factor of degree $d$. 

First, let us assume that there exists an $A\in \text{GL}(n,q)$ such that $A^M=\alpha$.  
Suppose, the combinatorial data $\Delta_A$ consists of polynomials $f_i$ of degree $d_i$ and 
partitions $\lambda_{f_i}$. Then, $M^{th}$ power of roots of $f_i$, for all $i$, are roots of 
$h(x)$, i.e., $(f_i)_M=h(x)$ for all $i$.  Therefore, by Proposition~\ref{M-power-single-poly}, the 
associated partition $\lambda_{h(x)}$ will have each $\lambda_{i_j}$ repeating $\frac{d_i}{d}$ many 
times. But, we are given that parts of $\lambda_{h(x)}$ are all distinct. Hence, $d_i=d$ for all 
$i$. Thus, by fixing $f$ as one of the $f_i$ and by using Lemma~\ref{power-poly}, we have $s=1$ and 
$h(x^M)$ has an irreducible factor of degree $d$, as required.

Now for converse, since $h(x^M)$ has an irreducible factor of degree $d$, call it $g(x)$. Then 
$M^{th}$ power of each root of $g(x)$ is a root of $h(x)$. Now, take $A$ to be the standard 
representative of the conjugacy class with combinatorial data $\Delta_A$ consisting of the 
polynomial $g(x)$ with $\lambda_{g(x)}=\lambda_{h(x)}$. From 
Proposition~\ref{M-power-partition-split}, we see that $\Delta_{A^M} = \Delta_{\alpha}$. This proves 
the required result.
\end{proof}

To obtain neat results for arbitrary $\alpha$ in $\text{GL}(n,q)$ we put some restrictions on $M$ 
(for example a prime power). Recall (last para of Section~\ref{M-power-polynomials}) that for $1\leq 
i\leq a$, we denote the set of all polynomials $f\in \Phi$ such that $f(x^M)$ has $r^{a-i}$ 
irreducible factors each of degree $r^i deg(f)$, by $\Phi_{M,i}$. Also, for convenience we denote 
the set of M-power polynomials $\Phi^M=\Phi_{M ,0}$ and we have, $\Phi=\displaystyle \bigcup_{b=0}^a 
\Phi_{M,b}$.
\begin{proposition}\label{prime-power-case}
Let $M=r^a$ where $r$ is a prime. Let $\alpha \in \text{GL}(n,q)$ with combinatorial data 
$\Delta_{\alpha} $ consisting of polynomials $f_i \in \Phi $ of degree $d_i$ and partitions 
$\lambda_{f_i}=1^{m_1(\lambda_{f_i})} \ldots j^{m_j(\lambda_{f_i})} \ldots$ written in power 
notation, $1\leq i \leq l$. Then, $X^M=\alpha$ has a solution in $\text{GL}(n,q)$ if and only if for 
each $1\leq i\leq l$, one of the following holds:
\begin{enumerate}
\item $f_i \in \Phi^M$.
\item $f_{i}\in \Phi_{M,b}$ for some $b$, $1\leq b \leq a$ and $r^b \mid m_j(\lambda_{f_i})$ 
for all $j$.
\end{enumerate}
\end{proposition}
\begin{proof} 
Let us first assume $X^M=\alpha$ has a solution $B$ in $\text{GL}(n,q)$. It is enough to prove 
this result when $\Delta_{\alpha}$ consists of a single irreducible polynomial $f$ with associated 
partition $\lambda_{f}$. Let the degree of $f$ be $d$ and hence $|\lambda_f|=\frac{n}{d}$. Since, 
$M$ is a prime power, from Lemma~\ref{M-is-a-power} either $f(x^M)$ is an M-power polynomial or it 
splits into $r^{a-b}$ irreducible polynomials each of degree $dr^b$ for some $b\geq 1$. That is, 
either $f\in \Phi^M$ or $f\in \Phi_{M,b}$. We show that if (2) does not hold then $f$ must be an 
M-power polynomial. Thus, let us assume that there exists $i_0$, such that $m_{i_0}(\lambda_f)$, the 
number of times $i_0$ appears in the partition $\lambda_f$, is not divisible by $r^b$. Now, we need 
to show that $f(x^M)$ has a factor of degree $d$. Let $\Delta_B$ consists of irreducible polynomials 
$g_1, g_2, \ldots $ with associated partitions $\lambda_{g_1}, \lambda_{g_2}, \ldots$. Since, 
$B^M=\alpha$ the $M^{th}$ power of roots of $g_j$ are roots of $f$ for all $j$. Then, from 
Proposition~\ref{M-power-single-poly}, we conclude that  $\Delta_{B^M}$ consists of the polynomial 
$f$ with the partition where each part of $\lambda_{g_j}$ repeats $\frac{s_j}{d}$ times, where 
$deg(g_j)=s_j$. Thus, $d\mid s_j$ for all $j$. Notice that a particular part in $\lambda_f$ can come 
from more than one $\lambda_{g_j}$, i.e, $m_{i_0}(f)$ is of the form $\sum_j \frac{s_j}{d}$. Now, 
from Lemma~\ref{power-poly} we see that $g_j$ are the factors of $f(x^M)$. Invoking 
Lemma~\ref{M-is-a-power}, each irreducible factor of $f(x^M)$ (which are $g_j$ in our case) has 
degree $dr^b$. Thus, $s_j=dr^b$. Since, $r^b \nmid m_{i_0}(f)$ there exists $j_0$ such that 
$r^b\nmid s_{j_0}$. Hence, $s_{j_0}=d$. This implies $f$ is an M-power polynomial. 
		
To prove the converse, we can work with the blocks of either kind. First, let $f\in \Phi^M$, 
i.e., $f(x^M)$ has an irreducible factor of degree $d$. Then, following the proof for converse of 
Proposition~\ref{M-power-determined}, we get a solution for $X^M=\alpha$. The main case we need to 
deal with is the second kind. Let $\alpha$ has associated data $\Delta_{\alpha}$ consisting of 
polynomial $f$ and partition $\lambda_f = 1^{m_1}\ldots i^{m_i}\ldots$ with the property that $f\in 
\Phi_{M,b}$ for some $b\geq 1$, i.e, $f(x^M)$ is a product of $r^{a-b}$ irreducible polynomials each 
of degree $dr^b$, and $r^b\mid m_{i}$ for all $i$. Let $g$ be one of the factors of $f(x^M)$ and 
$\lambda_{g} = 1^{\frac{m_1}{r^b}}\ldots \lambda_i^{\frac{m_i}{r^b}} \ldots$. Let $B$ be a matrix 
associated with data $g$ and $\lambda_{g}$. Then from Proposition~\ref{M-power-partition-split}, 
$B^M$ is conjugate to $\alpha$. This completes the proof. 
\end{proof}
\noindent Now, we write a corollary of this when $M$ is a prime. 
\begin{corollary}\label{prime-case}
Let $M$ be a prime with $(q,M)=1$. Denote $t=\mathfrak M(M;q)$. Let $\alpha \in \text{GL}(n,q)$ 
with combinatorial data $\Delta_{\alpha} $ consisting of polynomials $f_i \in \Phi $ of degree $d_i$ 
and partitions $\lambda_{f_i}=(\lambda_{i_1}, \lambda_{i_2}, \ldots)$, $1\leq i \leq l$. Then, 
$X^M=\alpha$ has a solution in $\text{GL}(n,q)$ if and only if for each $1\leq i\leq l$ one of the 
following holds,
\begin{enumerate}
\item $t\nmid d_i$.
\item $f_i\in \Phi^M$ (in this case, it is equivalent to saying that $f_i(x^M)$ is 
reducible).
\item $M \mid m_j(\lambda_{f_i})$ for every $j$.
\end{enumerate}
\end{corollary}
\noindent This follows from Lemma~\ref{M-is-prime}.

%%%%%%%%%%%%%%%%%%%%%
\section{$M^{th}$ power regular semisimple and regular classes in $\text{GL}(n,q)$}

In this section, we look at the regular and regular semisimple classes in $\text{GL}(n,q)$ which are 
$M^{th}$ powers and get generating function for the same. 
	\begin{proposition}\label{M-power-regular}
Let $\alpha\in \text{GL}(n,q)$ with associated data $\Delta_{\alpha}$. Let $\alpha$ be a regular 
element with the polynomials $f_1, \ldots, f_l$ in $\Delta_{\alpha}$.  Then, $X^M=\alpha$ has a 
solution in $\text{GL}(n,q)$ if and only if $f_i$ is M-power polynomial, for all $i$.
	\end{proposition}
\begin{proof}
Since, $\alpha$ is regular the associated partition $\lambda_{f_i}$ has single part, for all 
$i$. The result follows from Proposition~\ref{M-power-determined}.
\end{proof}
\noindent We note that if $\alpha$ is a regular semisimple element, we can apply this proposition as 
well. The generating functions are as follows.
\begin{theorem}\label{Theorem-rs-rgCC}
Let $M\geq 2$ be an integer and $(q, M)=1$. For the group $\text{GL}(n,q)$, the generating 
function for regular and regular semisimple classes which are $M^{th}$ power is,
\begin{enumerate}
\item $\displaystyle 1+ \sum_{n=1}^{\infty} c(n,q,M)_{\rg}u^n= \prod_{d\geq 
1}(1-u^d)^{-N_M(q,d)}$.
\item $\displaystyle 1+ \sum_{n=1}^{\infty} c(n,q,M)_{\rs}u^n= \prod_{d\geq 
1}(1+u^d)^{N_M(q,d)}$.
\end{enumerate}
\end{theorem}
\begin{proof}
From Proposition~\ref{M-power-regular}, it follows that a regular class $\alpha \in 
\text{GL}(n,q)$ is a $M^{th}$ power in $\text{GL}(n,q)$ if and only if each irreducible factor 
$f(x)$ of its characteristic polynomial $\chi_{\alpha}(x)$ is M-power polynomial. In other words, 
the regular conjugacy classes which are  $M^{th}$ power, are in one-one correspondence with the set 
of M-power polynomials with non-zero constant term. 
Therefore,
$$ 1+ \sum_{n=1}^{\infty} c(n,q,M)_{\rg}u^n= \prod_{f\in \Phi^M} (1-u^{deg(f)})^{-1}= 
\prod_{d\geq 1} (1-u^d)^{-N_{M}(q,d)}. $$
This proves the first part.

The regular semisimple $M^{th}$ power conjugacy classes in $\text{GL}(n,q)$ are characterized by 
separable M-power polynomials with non-zero constant term, and hence, 
$$1+\sum_{n=1}^{\infty} c(n,q,M)_{\rs}u^n= \prod_{f\in \Phi^M} (1+u^{deg(f)})= \prod_{d\geq 
1}(1+u^d)^{N_M(q,d)}.$$
This proves the required result.
\end{proof}
\noindent Now, we can use this to get the generating function for the $M^{th}$ power regular and 
regular semisimple elements. 
\begin{theorem}\label{Theorem-rs-rgAE}
For the group $\text{GL}(n, q)$, and $M\geq 2$ with the condition that $(q,M)=1$, 
\begin{enumerate}
\item the generating function for the regular semisimple elements which are $M^{th}$ power 
is 
$$\displaystyle 1+ \sum_{n=1}^{\infty} \frac{|\text{GL}(n,q)_{\rs}^M|}{|\text{GL}(n,q)|}u^n= 
\prod_{d\geq 1} \left(1+\frac{u^d}{q^d-1} \right)^{N_M(q,d)}.$$
\item The generating function for the regular elements which are $M^{th}$ power is  
\begin{eqnarray*} 1+ \sum_{n=1}^{\infty} 
\frac{|\text{GL}(n,q)_{\rg}^M|}{|\text{GL}(n,q)|}u^n &=& \prod_{d\geq 1} \left(1+\sum_{j=1}^{\infty} 
\frac{u^{jd}}{q^{(j-1)d}(q^d-1)} \right)^{N_M(q,d)} \\ &=& \prod_{d\geq 1} \left(1-\frac{u^d}{q^d} 
\right)^{-N_M(q,d)} \prod_{d\geq 1} \left(1+\frac{u^d}{q^d(q^d-1)} \right)^{N_M(q,d)}.
\end{eqnarray*}
\end{enumerate}
\end{theorem}
\begin{proof}
We use Proposition~\ref{M-power-determined} here. To get (1), in the Equation~\ref{cycle-index} 
of cycle index generating function,  we take $n=1$ on the right side (and hence the second sum runs 
over partitions of $1$ which is $(1)$) and the outer product runs over all $f\in \Phi^M$. Thus, to 
get the desired generating function we put $x_{f,\lambda}=1$, when $f\in \Phi^M$ and $0$ otherwise. 
We get,
$$1+ \sum_{n=1}^{\infty} \frac{|\text{GL}(n,q)_{\rs}^M|}{|\text{GL}(n,q)|} u^n = \prod_{f\in 
\Phi^M}\left(1+\frac{u^{deg(f)}}{q^{deg(f)}-1}\right) = \prod_{d\geq 1} \left(1+\frac{u^d}{q^d-1} 
\right)^{N_{M}(q,d)}.$$
Here we used the following: for the partition $(1)=1^1$ and $q^{deg(f). 
\sum_i(\lambda'_i)^2}\left( \frac{1}{q^{deg(f)}}\right)_1=  q^{deg(f)}\left(1-  
\frac{1}{q^{deg(f)}}\right) =  q^{deg(f)}-1$.

The generating function for regular elements is obtained in similar fashion. Here we take the 
partition $(n) \vdash n$ on the right in the cycle index generating function. The transpose of this 
partition is $(n)'= (1, 1, \ldots, 1)=1^n$ and hence $q^{deg(f). \sum_i(\lambda'_i)^2}\left( 
\frac{1}{q^{deg(f)}}\right)_1 =  q^{n.deg(f)}\left(1-  \frac{1}{q^{deg(f)}}\right) =  
q^{(n-1).deg(f)} (q^{deg(f)}-1)$.
Therefore,
\begin{eqnarray*}
1+ \sum_{n=1}^{\infty}  \frac{|\text{GL}(n,q)_{\rg}^M|}{|\text{GL}(n,q)|} u^n   &=&  
\prod_{f\in \Phi^M}  \left( 1 + \sum_{j=1}^{\infty}  \frac{u^{j.deg(f)}}{q^{(j-1) deg(f)}(q^{deg(f)} 
-1 )}\right)\\
&= & \prod_{d\geq 1} \left(1+\sum_{j=1}^{\infty} \frac{u^{jd}}{q^{(j-1)d}(q^d-1)} 
\right)^{N_{M}(q, d)}.
\end{eqnarray*}
To deduce the alternate formula, we note that,
$$1+\sum_{j=1}^{\infty} \frac{u^{jd}}{q^{(j-1)d}(q^d-1)} = 
\left(1-\frac{u^d}{q^d}\right)^{-1}\left(1+\frac{u^d}{q^d(q^d-1)}\right)$$
which can be verified by computing coefficients on both sides.
\end{proof}

%%%%%%%%%%%%%%%%%%%%%
\section{$M^{th}$ power semisimple classes in $\text{GL}(n,q)$ when $M$ is a prime power}
	
In this section, we deal with semisimple elements which are $M^{th}$ power. We assume $M=r^a$ for 
some prime $r$ and $(q,M) = 1$.  
\begin{proposition}\label{M-power-ssclass}
Let $M=r^a$ be a prime power and $(q,M)=1$. Let $\alpha\in \text{GL}(n, q)$ be semisimple with 
the corresponding combinatorial data $\Delta_{\alpha}$ consisting of polynomials $f_i$ and 
partitions $\lambda_{f_i}$. Then, $X^M=\alpha$ has a solution in $\text{GL}(n, q)$ if and only if 
for each $i$, one of the following holds,
\begin{enumerate}
\item $f_i\in \Phi^M$.
\item $f_i\in \Phi_{M,b}$,  for some $1\leq b \leq a$, and $r^b \mid |\lambda_{f_i}|$.  
\end{enumerate}
\end{proposition}
\begin{proof}
We recall that when $\alpha$ is semisimple all partitions in $\Delta_{\alpha}$ are of the form 
$1^{|\lambda_{f_i}|}$. Thus, the second condition in Proposition~\ref{prime-power-case} becomes the 
required one here.
\end{proof}
\noindent Now recall the notation $N_M^i(q, d)$ preceding the Proposition~\ref{nmiqd}. We have,    
\begin{theorem}\label{gen-fun-primepower}
Let $M=r^a$ be a prime power and $(q,M)=1$. Then, we have the following generating functions:
\begin{enumerate}
\item $\displaystyle 1+ \sum_{n=1}^{\infty} c(n,q,M)_{\s} u^n = \prod_{i=0}^{a}\prod_{d\geq 
1} \left(1-u^{r^id}\right)^{-N_M^i(q,d)}$.
\item $1+\displaystyle 
\sum_{n=1}^{\infty}\frac{|\text{GL}(n,q)_{\s}^M|}{|\text{GL}(n,q)|}u^n = \prod_{i=0}^{a}\prod_{d\geq 
1} \left(1+ \sum_{j=1}^{\infty} \frac{u^{r^ijd}}{{q^{ \frac{r^ij(r^ij -1)d}{2}} } 
\prod_{t=1}^{r^ij}(q^{td}-1)} \right)^{N_M^i(q,d)}$.
\end{enumerate}
\end{theorem}
\begin{proof}
Recall the notation $\Phi_{M, i}$ defined at the end of Section~\ref{M-power-polynomials} when 
$M=r^a$. By Proposition~\ref{prime-power-case}, it is clear that a semisimple conjugacy class which 
is $M^{th}$ power, corresponds to (in fact, one-one correspondence) a monic polynomial $g$ of degree 
$n$ over $\mathbb{F}_q$ with the property that the multiplicity of each of its irreducible factors 
which belong to $\Phi_{M, i}$ for some $i$, must be a multiple of $r^i$. Therefore, we get,
\begin{eqnarray*}
1 + \sum_{n=1}^{\infty} c(n,q,M)_{\s}u^n &=& \prod_{i=0}^{a} \prod_{f\in \Phi_{M,i}} 
\left(1+ u^{r^i deg(f)} + u^{2r^i deg(f)}+ \cdots \right)  \\ 
&=& \prod_{i=0}^{a}\prod_{f\in \Phi_{M, i}} \left( 1 - u^{r^i deg(f)} \right)^{-1} = 
\prod_{i=0}^{a} \prod_{d\geq 1} \left(1-u^{r^id}\right)^{-N_M^i(q,d)}. 
\end{eqnarray*}
\noindent This proves the first part.

For the proof of second part, we use the cycle index generating function once again. In the 
Equation~\ref{cycle-index}, on the right hand side, we put $x_{f, \lambda}=1$ when $\lambda=(1, 1, 
\ldots, 1) \vdash r^ij$, and $f\in \Phi_{M, i}$ for each $j\geq 1$, else we put $x_{f, \lambda}=0$.  
We also note that when $\lambda=(1, 1, \ldots, 1) \vdash n$ and $f\in \Phi$, we have,
\begin{eqnarray*}
q^{deg(f).\sum_{i} (\lambda^{'}_{i})^2}  \prod_{i\geq 1} 
\left(\frac{1}{q^{deg(f)}}\right)_{m_i(\lambda)} &=& q^{n^2 deg(f)} 
\left(1-\frac{1}{q^{deg(f)}}\right) \cdots \left(1-\frac{1}{q^{n.deg(f)}}\right) \\ 
&=& q^{n^2.deg(f)}\frac{(q^{deg(f)}-1) \cdots (q^{n. deg(f)}-1)}{q^{\frac{n(n+1)}{2} 
deg(f)}} \\
&=& q^{\frac{n(n-1)}{2} deg(f)} \prod_{i=1}^{n}(q^{i .deg(f)}-1).
\end{eqnarray*}
Therefore, we have,
$$1+\sum_{n=1}^{\infty} \frac{|\text{GL}(n,q)_{\s}^M|}{|\text{GL}(n,q)|} u^n = 
\prod_{i=0}^{a}\left( \prod_{f\in \Phi_{M, i}} \left( 1+ \sum_{j=1}^{\infty} \frac{u^{r^ij. 
deg(f)}}{{q^{\frac{r^ij(r^ij-1). deg(f)}{2}}} \prod_{t=1}^{r^ij}(q^{t.deg(f)}-1)} \right)\right).$$
This gives the desired generating function.
\end{proof}

In the case when $M=r$, a prime, the formula gets further simplified as $i=0$ and $1$ in the 
formula above. 
\begin{corollary}
Let $M$ be a prime and $(q,M)=1$. Let $\alpha\in \text{GL}(n,q)$ be semisimple with the 
corresponding combinatorial data $\Delta_{\alpha}$ consisting of polynomials $f_i$ and partitions 
$\lambda_{f_i}$. Then, $X^M=\alpha$ has a solution in $\text{GL}(n,q)$ if and only if for each $i$, 
one of the following holds,
\begin{enumerate}
\item $f_i\in \Phi^M$.  
\item $M\mid |\lambda_{f_i}|$.
\end{enumerate}
\end{corollary}
\begin{proof}
This follows from Proposition above and Corollary~\ref{prime-case} .
\end{proof}
\begin{corollary}
Let $M$ be a prime with $(q,M)=1$. Then, 
\begin{eqnarray*}
\displaystyle 1+\sum_{n=1}^{\infty}c(n,q,M)_{\s}u^n = \left(\frac{1-u^M}{1-qu^M}\right) 
\prod_{d\geq 1}\left(1+u^d+u^{2d} + \cdots +u^{d(M-1)}\right)^{N_M(q,d)}.
\end{eqnarray*}
\end{corollary}
\begin{proof}
Recall that here we have $N(q, d) = N_M^0(q, d) + N_M^1(q, d) = N_M(q, d) + N_M^1(q, d)$. 
By taking $a=1$ (and thus $r=M$) in Theorem~\ref{gen-fun-primepower}, we have, 
\begin{eqnarray*}
&&1+\sum_{n=1}^{\infty}c(n,q,M)_{\s}u^n = \prod_{d\geq 1}(1-u^{d})^{-N_M(q,d)} \prod_{d\geq 
1}(1 - u^{Md})^{-N_M^1(q,d)}  \\ 
&=& \prod_{d\geq 1}(1-u^{d})^{-N_M(q,d)} \prod_{d\geq 1}(1 - u^{Md})^{N_M(q,d)-N(q, d)}  \\
&=& \prod_{d\geq 1}\left(\frac{1-u^{Md}}{1-u^d} \right)^{N_M(q,d)} \prod_{d\geq 1}(1 - 
u^{Md})^{-N(q, d)}  =  \prod_{d\geq 1}\left(\frac{1-u^{Md}}{1-u^d} \right)^{N_M(q,d)}. 
\left(\frac{1-u^M}{1-qu^M}\right).
\end{eqnarray*}
The last equality follows from the generating function formula for $N(q, d)$ (see the 
Equation~\ref{number-rgclasses} by taking $u^M$ for $u$). 
\end{proof}

%%%%%%%%%%%%%%%%%%%%%%%%%%%%%%%%%%%%%%%%%%%%%%
\section{$M^{th}$ power conjugacy classes in $\text{GL}(n, q)$ when $M$ is a prime power}

In this section, we work with general elements and assume $M=r^a$, for some prime $r$, and 
$(q,M)=1$. Now, we proceed to construct generating functions for $c(n,q,M)$. For this, we use the 
description of conjugacy classes given by Macdonald~\cite{Ma} which is slightly different from the 
one what we have been using so far, but more convenient for counting. We first define what are known 
as type-$\nu$ conjugacy classes in $\text{GL}(n,q)$, for some partition $\nu\vdash n$. We follow 
Macdonald's (see~\cite{Ma}) exposition here.
\begin{theorem}[\cite{Ma} 1.8, 1.9]
Let $C$ be a conjugacy class of $\text{GL}(n,  q)$. Then, to $C$, we can associate a sequence of 
polynomials $(u_1, u_2, \ldots)$, with the following properties:
\begin{enumerate}
\item $u_i\in \mathbb{F}_q[x]$ with $u_i(0)=1$ for each $i$, and
\item $\displaystyle\sum_{i\geq 1} i  \mathrm{deg}(u_i)=n$.
\end{enumerate}
This data determines $C$ uniquely, and hence gives a one-one correspondence between the 
conjugacy classes of $\text{GL}(n, q)$ with the sequence of polynomials satisfying the two 
properties.
\end{theorem}
In Section~\ref{Scycle-index}, to an element $A$ in a conjugacy class $C$ of $\text{GL}(n,q)$, we 
associated a combinatorial data $\Delta_A$, which consists of polynomials $f_j\in \Phi$ and 
partitions $\lambda_{f_j}$ such that $\displaystyle\sum_{j} |\lambda_{f_j}| \mathrm{deg}(f_j)=n$. 
The relation between this data to that of Macdonald's is as follows. Define,
$$u_i(x) = \delta_i  \prod_{j} f_j(x)^{m_i(\lambda_{f_j})}$$
where $\delta_i$ is chosen such that $u_i(0)=1$ and the notation $m_i(\lambda_{f_j})$ is the number 
of times $i$ appears in the partition $\lambda _{f_j}$. Clearly, there are finitely many sequence of 
such non-constant polynomials $u_i$ satisfying the equation $\displaystyle\sum_{i} i 
\mathrm{deg}(u_i)=n$. 
	
Now, given a partition $\nu = 1^{n_1}2^{n_2}\ldots $ of $n$, we say that a {\bf conjugacy class $C$ 
is of type-$\nu$}, if the associated sequence of polynomials, as per Macdonald, $(u_1, u_2, \ldots)$ 
satisfy $\mathrm{deg}(u_i)=n_i$. In terms of the combinatorial data $\Delta_C$, we see that the 
conjugacy class $C$ is of type-$\nu$ if $\displaystyle\sum_{j}m_i(\lambda_{f_j})=n_i$, for all $i$. 
For example, when $\nu=1^n\vdash n$ the conjugacy class of type-$\nu$ has a single polynomial $u_1$ 
of degree $n$, and it corresponds to a semisimple conjugacy class. Let $c_{\nu}$ be the number of 
conjugacy classes of type-$\nu$. Then, $c_{\nu}= \prod_{n_i>0} (q^{n_i}-q^{n_i-1})$, because the 
number of polynomials $u_i\in \mathbb{F}_q[x]$ of degree $n_i$, satisfying $u_i(0)=1$, is 
$q^{n_i}-q^{n_i-1}$. Therefore, we have the number of conjugacy classes in $\text{GL}(n ,q)$ is 
$c(n)=\displaystyle\sum_{\nu \vdash n}c_{\nu} = \sum_{\nu} \prod_{n_i>0} (q^{n_i}-q^{n_i-1})$. This 
gives the generating function for $c(n)$ which is the Equation~\ref{number-cclasses}. 
Now, we determine the number of type-$\nu$ conjugacy classes that are $M^{th}$ powers.  Recall the 
notation: $\Phi_{M, i}$ is the set of all polynomials $f\in \Phi$ with the property that all 
irreducible factors of $f(x^M)$ are of degree $r^i deg(f)$. We also have $\Phi=\bigcup_{i=0}^a 
\Phi_{M, i}$ where $\Phi_{M, 0}$ the set of M-power polynomials.  The 
Proposition~\ref{prime-power-case} can be rephrased in terms of the Macdonald's notation as follows.
\begin{proposition}\label{M-power-cclasses}
Let $M=r^a$ where $r$ is a prime and $(q,M)=1$. Let $\alpha \in \text{GL}(n, q)$, with 
associated Macdonald's data $(u_1, u_2, \ldots)$. Write $u_i(x) = \kappa_i \prod_{j}f_{ij}^{a_{ij}}$ 
as a product of irreducible polynomials $f_{ij}\in \Phi$ and $\kappa_i\in \mathbb F_q$ to make 
$u_i(0)=1$. Then, $X^M=\alpha$ has a solution in $\text{GL}(n, q)$ if and only if, for all $f_{ij}$, 
$f_{ij} \in \Phi_{M, b}$, for some $0\leq b \leq a$, implies $r^b \mid a_{ij}$. 
	\end{proposition}
\begin{proof}
We write each $u_i(x)=\kappa_i \prod_{j}f_{ij}^{a_{ij}}$ as a product of irreducibles. Then the 
set $f_{ij}$ and the corresponding powers $m_i(\lambda_{f_{ij}})=a_{ij}$ give back the combinatorial 
data $\Delta_{\alpha}$. The result follows from Proposition~\ref{prime-power-case}. 
	\end{proof}
Note the subtle difference between this proposition and the semisimple case 
(Proposition~\ref{M-power-ssclass}). In the present case, we require that $r^b$ divides multiplicity 
of each part appearing in the partitions. In general, it is not true that $X^M=\alpha$ has a 
solution in $\text{GL}(n,q)$ if and only if $Y^M=\alpha_s$ has a solution where $\alpha_s$ is the 
semisimple part of $\alpha$. 
	\begin{example}
Take $\alpha=\begin{psmallmatrix} \lambda_1 & 1 & \\ &\lambda_1 & \\ 
&&\lambda_2\end{psmallmatrix}\in \text{GL}(3,q)$ and $M=2$. Then, $X^M=\alpha$ has a solution in 
$\text{GL}(3,q)$ if and only if $\lambda_1, \lambda_2\in {\mathbb F_q^*}^2$. However, $Y^2= 
\alpha_s=\begin{psmallmatrix} \lambda_1 &  & \\ &\lambda_1 & \\ & & \lambda_2 \end{psmallmatrix}$ 
has solution if and only if $\lambda_2 \in {\mathbb F_q^*}^2$, because $\begin{psmallmatrix} & 
\lambda_1 \\ 1& \end{psmallmatrix}^ 2 = \begin{psmallmatrix} \lambda_1 & \\ &\lambda_1 
\end{psmallmatrix}$.
	\end{example}
Now, we write the generating function for $c(n,q,M)$. We begin with (see~\cite[Lemma 2]{FF}),
\begin{lemma}
Let $f(u)=1+\sum_{n=1}^{\infty}a_n u^n$. Suppose $\nu =1^{n_1}2^{n_2}\ldots$ is a partition of 
$n$. Define $b_n = \displaystyle \sum_{\nu \vdash n} \left(\prod_{n_i > 0} a_{n_i}\right)$. Then, 
$$1 + \sum_{n=1}^{\infty} b_n u^n = \prod_{t=1}^{\infty} f(u^t).$$
\end{lemma}
\begin{proof}
The Lemma follows simply by computing the coefficients of $u^n$ on both sides.
\end{proof}
\noindent  We have the following,
\begin{theorem}\label{Theorem-AE}
Let $M=r^a$, where $r$ is a prime, and $(q,M)=1$. Then we have the following generating 
function,
$$1 + \sum_{n=1}^{\infty} c(n,q,M) u^n = \prod_{j=1}^{\infty} \prod_{i=0}^{a} \prod_{d\geq 1} 
(1-u^{jr^id})^{-N_M^i(q,d)}.$$
\end{theorem}
\begin{proof}
Let $c_{\nu, M}$ denote the number of type-$\nu$ conjugacy classes that are $M^{th}$-powers.
For a partition $\nu=1^{n_1}2^{n_2}\ldots$ of $n$, from Proposition~\ref{M-power-cclasses} we 
have 
$$c_{\nu,M}=\prod\limits_{n_i>0}c(n_i,q,M)_{\s}$$
where $n_i$ represent $deg(u_i)$.
Now, 
$$ c(n,q,M)=\sum_{\nu\vdash n} c_{\nu,M} = \sum_{\nu\vdash n} 
\left(\prod\limits_{n_i>0}c(n_i,q,M)_{\s} \right).
$$
We apply the previous Lemma by taking $a_n = c(n,q,M)_{\s}$, thus 
$f(u)=\prod_{i=0}^{a}\prod_{d\geq 1} (1-u^{r^id})^{-N_M^i(q,d)}$ is the generating function for 
$M$-power semisimple classes (by Theorem~\ref{gen-fun-primepower}). Thus, $b_n=c(n,q,M)$ and we get,
$$1+\sum\limits_{n=1}^{\infty}c(n,q,M)u^n = \prod_{t=1}^{\infty}f(u^t)$$
which gives the required result.
\end{proof}
\begin{corollary}\label{gen-function-all-classes-prime}
Let $M$ be a prime and $(q,M)=1$. Then we have,
$$ 1 + \sum_{n=1}^{\infty} c(n,q,M) u^n = \prod_{j=1}^{\infty}\left( 
\left(\frac{1-u^{Mj}}{1-qu^{Mj}}\right)\prod_{d\geq 1}(1+u^{jd}+u^{2jd}+\ldots 
+u^{jd(M-1)})^{N_M(q,d)}\right).$$
\end{corollary}

%%%%%%%%%%%%%%%%%%%%%%%%%%%%%%
\section{Generating functions when $M$ is a prime and $(M,q)=1$}\label{generating-function-prime}

When $M$ is a prime and $(M,q)=1$, several generating functions seen in the earlier sections can be 
further simplified. We do that in this section. 
Recall that $\mathfrak{M}(M,q)$ is the order of $q$ in $\Z/M\Z^{\times}$. In this section we will 
denote $\mathfrak{M}(M,q)$ by $t$ to make formula look cleaner. 
We begin with some identities.
\begin{lemma}\label{recursive-reln}
Suppose $k\geq 1$ and $M\nmid d$. Then,
$$N(q^{M^k},d) = M^kN(q,M^kd)+N(q^{M^{k-1}},d).$$
\end{lemma}
\begin{proof}
We have,
\begin{eqnarray*}
N(q^{M^k},d) &=& \frac{1}{d}\sum_{r\mid d}\mu(r)((q^{M^k})^{d/r}-1) \\ &=& 
\frac{1}{d}\left(\sum_{r\mid M^kd}\mu(r)(q^{\frac{M^kd}{r}}-1)-\sum_{\substack{Mr\mid M^kd \\ 
(r,M)=1}}\mu(Mr)(q^{\frac{M^kd}{Mr}}-1)\right) \\ &=& \frac{1}{d}\left(\sum_{r\mid 
M^kd}\mu(r)(q^{\frac{M^kd}{r}}-1)+\sum_{r\mid d}\mu(r)((q^{M^{k-1}})^{d/r}-1)\right)\\ &=& 
M^k\left(\frac{1}{M^kd}\sum_{r\mid 
M^kd}\mu(r)(q^{\frac{M^kd}{r}}-1)\right)+\left(\frac{1}{d}\sum_{r\mid 
d}\mu(r)((q^{M^{k-1}})^{d/r}-1)\right) \\ &=& M^kN(q,M^kd)+N(q^{M^{k-1}},d).
\end{eqnarray*}
This completes the proof.
\end{proof}
\begin{lemma}\label{gen-fun-nhatqd}
Let $M\geq 2$ be a prime with $(M,q)=1$. Let $t$ be the order of $q$ in $\Z/M\Z^{\times}$. Then,
$$\displaystyle \prod\limits_{d\geq 
1}(1-u^d)^{-N_M(q,d)}=\left(\frac{1-u}{1-qu}\right)\prod_{k=0}^{\infty}\left(\frac{1-u^{M^kt}}{
1-q^tu^{M^kt}}\right)^{\frac{1-M}{M^{k+1}t}}=c(q,u)\prod\limits_{k=0}^{\infty} 
c(q^t,u^{M^kt})^{\frac{1-M}{M^{k+1}t}}.$$
\end{lemma}
\begin{proof} From Equation~\ref{number-rgclasses},
\begin{eqnarray*} 
\displaystyle \prod_{d\geq 1}(1-u^d)^{-N_M(q,d)} &=& \prod_{d\geq 
1}(1-u^d)^{-N(q,d)}\prod_{d\geq 1}(1-u^d)^{N(q,d)-N_M(q,d)} \\ &=& c(q,u)\prod_{d\geq 
1}(1-u^d)^{N(q,d)-N_M(q,d)}.
\end{eqnarray*}
Now, using Corollary~\ref{final-reln-nmqd-nqd}, 
\begin{eqnarray*}
\prod\limits_{d\geq 1}(1-u^d)^{N(q,d)-N_M(q,d)} &=& \prod_{k=0}^{\infty} \prod_{M\nmid 
d}(1-u^{M^ktd})^{N(q,M^ktd)-N_M(q,M^ktd)} 
\end{eqnarray*}
\begin{eqnarray*}
\displaystyle &=& \prod\limits_{k=0}^{\infty}\prod\limits_{M\nmid 
d}(1-u^{M^ktd})^{\frac{M-1}{M^{k+1}t}N(q^{M^ktd},d)} = 
\left[\prod\limits_{k=0}^{\infty}\prod\limits_{M\nmid 
d}(1-u^{M^ktd})^{\frac{N(q^{M^ktd},d)}{M^k}}\right]^{\frac{M-1}{Mt}} \\ &=& 
\left[\prod\limits_{M\nmid 
d}(1-u^{td})^{N(q^t,d)}\right]^{\frac{M-1}{Mt}}\left[\prod\limits_{k=1}^{\infty}\prod\limits_{M\nmid 
d}(1-u^{M^ktd})^{\frac{N(q^{M^ktd},d)}{M^k}}\right]^{\frac{M-1}{Mt}} \\ 
&=& \displaystyle \left[\prod\limits_{M\nmid 
d}(1-u^{td})^{N(q^t,d)}\right]^{\frac{M-1}{Mt}}\left[\prod\limits_{k=1}^{\infty}\prod\limits_{M\nmid 
d}(1-u^{M^ktd})^{N(q^t,M^ktd)+\frac{N(q^{M^{k-1}t},d)}{M^k}}\right]^{\frac{M-1}{Mt}}.
\end{eqnarray*}
The last equality follows from Lemma~\ref{recursive-reln}, where we replace $q$ by $q^t$.
Therefore, we get
\begin{eqnarray*}
&&\prod\limits_{d\geq 1}(1-u^d)^{N(q,d)-N_M(q,d)}\\
&=& 
\left[\prod\limits_{k=0}^{\infty}(1-u^{M^ktd})^{-N(q^t,M^ktd)}\right]^{\frac{1-M}{Mt}}\left[
\prod\limits_{k=1}^{\infty}\prod\limits_{M\nmid 
d}(1-u^{M^ktd})^{\frac{N(q^{M^{k-1}t},d)}{M^k}}\right]^{\frac{M-1}{Mt}} \\ &=& \left(\prod_{d\geq 1} 
(1-u^{td})^{-N(q^t,td)} \right)^{\frac{1-M}{Mt}} 
\left[\prod\limits_{k=1}^{\infty}\prod\limits_{M\nmid 
d}(1-u^{M^ktd})^{\frac{N(q^{M^{k-1}t},d)}{M^{k-1}}}\right]^{\frac{M-1}{M^2t}} \\ &=& 
\left(\frac{1-u^t}{1-q^tu^t}\right)^{\frac{1-M}{Mt}}\left[\prod\limits_{k=1}^{\infty}\prod\limits_{
M\nmid d}(1-u^{M^ktd})^{\frac{N(q^{M^{k-1}t},d)}{M^{k-1}}}\right]^{\frac{M-1}{M^2t}}.
\end{eqnarray*}
Again, applying Lemma~\ref{recursive-reln}, and following the steps as above we get,
\begin{eqnarray*}
&&\prod\limits_{d\geq1}(1-u^d)^{N(q,d)-N_M(q,d)}\\
&=& \left(\frac{1-u^t}{1-q^tu^t}\right)^{\frac{1-M}{Mt}}\left(\frac{1-u^
{Mt}}{1-q^tu^{Mt}}\right)^{\frac{1-M}{M^2t}}\left[\prod\limits_{k=2}^{\infty}\prod\limits_{M\nmid 
d}(1-u^{M^ktd})^{\frac{N(q^{M^{k-2}t},d)}{M^{k-2}}}\right]^{\frac{M-1}{M^3t}}.
\end{eqnarray*}
Inductively,  we conclude
\begin{eqnarray*}
\prod\limits_{d\geq 1}(1-u^d)^{N(q,d)-N_M(q,d)} &=& 
\prod_{k=0}^{\infty}\left(\frac{1-u^{M^kt}}{1-q^tu^{M^kt}}\right)^{\frac{1-M}{M^{k+1}t}} = 
\prod\limits_{k=0}^{\infty} c(q^t,u^{M^kt})^{\frac{1-M}{M^{k+1}t}}.
\end{eqnarray*}
This completes the proof.
\end{proof}
%%%%%%%%%%%%%%
\subsection{Generating function for regular and regular semisimple conjugacy classes}

We now simplify Theorem~\ref{Theorem-rs-rgCC} when $M$ is a prime.
\begin{theorem}\label{gen_func_reg_regsem_class_revisited}
Let $M\geq 2$ be a prime and $(M,q)=1$. Let $t$ be the order of $q$ in $\Z/M\Z^{\times}$. Then,
\begin{enumerate}
\item $1+\sum\limits_{n=1}^{\infty}c(n,q,M)_{\rg}u^n=c(q,u)\prod\limits_{k=0}^{\infty} 
c(q^t,u^{M^kt})^{\frac{1-M}{M^{k+1}t}}$,
\item $1+\sum\limits_{n=1}^{\infty}c(n,q,M)_{\rs}u^n=s(q,u)\prod\limits_{k=0}^{\infty} 
s(q^t,u^{M^kt})^{\frac{1-M}{M^{k+1}t}}$
\end{enumerate}
where $c(q,u)$ and $s(q,u)$ are given by Equation~\ref{number-rgclasses} and, 
Equation~\ref{number-rsclasses} respectively.
\end{theorem}
\begin{proof} (1) follows from Theorem~\ref{Theorem-rs-rgCC} and Lemma~\ref{gen-fun-nhatqd}.

For (2), from Theorem~\ref{Theorem-rs-rgCC} we have,
\begin{eqnarray*}
1+\sum\limits_{n=1}^{\infty}c(n,q,M)_{\rs}u^n &=& \prod_{d\geq 
1}(1+u^d)^{N_M(q,d)}=\frac{\prod\limits_{d\geq 1}(1-u^{2d})^{N_M(q,d)}}{\prod\limits_{d\geq 
1}(1-u^d)^{N_M(q,d)}} \\ &=& \frac{c(q,u)\prod\limits_{k=0}^{\infty} 
c(q^t,u^{M^kt})^{\frac{1-M}{M^{k+1}t}}}{c(q,u^2)\prod\limits_{k=0}^{\infty} 
c(q^t,u^{2M^kt})^{\frac{1-M}{M^{k+1}t}}}=s(q,u)\prod\limits_{k=0}^{\infty} 
s(q^t,u^{M^kt})^{\frac{1-M}{M^{k+1}t}}.
\end{eqnarray*}
The last equality follows since $\displaystyle s(q,u)=\frac{c(q,u)}{c(q,u^2)}.$
\end{proof}

We give some examples how to explicitly get the coefficients once we know the generating functions. 
One can use some computer algebra system, such as SAGEMATH to do this.
\begin{example}
Let us take $M=3$, and $q\equiv 2 \imod 3$. Therefore, $t=2$ and by the theorem above we have
\begin{eqnarray*}
1+\sum_{n=1}^{\infty}c(n,q,3)_{\rg}u^n &=& c(q,u)\prod\limits_{k=0}^{\infty} 
c(q^2,u^{2.3^k})^{-\frac{1}{3^{k+1}}}\\ &=& 
\left(\frac{1-u}{1-qu}\right)\left(\frac{1-u^2}{1-q^2u^2}\right)^{-\frac{1}{3}}\left(\frac{1-u^6}{
1-q^2u^6}\right)^{-\frac{1}{9}}\prod\limits_{k=2}^{\infty} c(q^2,u^{2.3^k})^{-\frac{1}{3^{k+1}}}.
\end{eqnarray*}
\noindent Similarly, 
\begin{eqnarray*}
&& 1+\sum_{n=1}^{\infty}c(n,q,3)_{\rs}u^n = s(q,u)\prod\limits_{k=0}^{\infty} 
s(q^2,u^{2.3^k})^{-\frac{1}{3^{k+1}}} \\ &=& \left(\frac{1-qu^2}{(1+u)(1-qu)}\right)
\left(\frac{1-q^2u^4}{(1+u^2)(1-q^2u^2)}\right)^{-\frac{1}{3}}\left(\frac{1-q^2u^{12}}{
(1+u^6)(1-q^2u^6)}\right)^{-\frac{1}{9}}\prod\limits_{k=2}^{\infty} 
s(q^2,u^{2.3^k})^{-\frac{1}{3^{k+1}}}.
\end{eqnarray*}
\noindent We make a table which records $c(n,q,3)_{\rg}$ for some small values of $n$. We 
know, $c(n,q)_{\rg}=q^n-q^{n-1}$.
\vspace{0.1 cm}
\begin{table}[ht]
\renewcommand{\arraystretch}{1.2}
\centering
\begin{tabular}{|c|c|c|}
\hline 
$n$ & Number of regular classes & Number of $3^{rd}$ power regular classes\\
\hline
1 & $q-1$ & $q-1$\\ \hline
2 & $q^2-q$ & $\frac{2}{3}q^2-q+\frac{1}{3}$\\ \hline
3 & $q^3-q^2$ & $\frac{2}{3}q^3-\frac{2}{3}q^2+\frac{1}{3}q-\frac{1}{3}$\\ \hline
4 & $q^4-q^3$ & $\frac{5}{9}q^4-\frac{2}{3}q^3+\frac{2}{9}q^2 -\frac{1}{3}q+\frac{2}{9}$\\ \hline
5 & $q^5-q^4$ & $\frac{5}{9}q^5-\frac{5}{9}q^4+\frac{2}{9}q^3 
-\frac{2}{9}q^2+\frac{2}{9}q-\frac{2}{9}$\\ \hline
6 & $q^6-q^5$ & 
$\frac{40}{81}q^6-\frac{5}{9}q^5+\frac{5}{27}q^4-\frac{2}{9}q^3+\frac{1}{27}q^2-\frac{2}{9}q+\frac{ 
23}{81}$\\ \hline
\end{tabular}
\vspace{.5 cm}
\caption{\label{table1} Table for $c(n,q,3)_{\rg}$ when $q\equiv 2 \imod 3$.}
\end{table}
\noindent We now make a table which records $c(n,q,3)_{\rs}$ for some small values of $n$ along 
with $c(n,q)_{\rs}$ (see the formula before Equation~\ref{number-rsclasses}) for comparisons. 

\vskip5mm 
\begin{longtable}{|c|c|c|}
%\renewcommand{\arraystretch}{1.2}
%\centering
%\begin{tabular}{|c|c|c|}
\hline 
$n$ & Number of regular semisimple classes & Number of $3^{rd}$ power regular semisimple classes\\ 
\hline
1 & $q-1$ & $q-1$\\ \hline
2 & $q^2-2q+1$ & $\frac{2}{3}q^2-2q+\frac{4}{3}$\\ \hline
3 & $q^3-2q^2+2q-1$ & $\frac{2}{3}q^3-\frac{5}{3}q^2+\frac{7}{3}q-\frac{4}{3}$\\ \hline
4 & $q^4-2q^3+2^2q-2q+1$ & $\frac{5}{9}q^4-\frac{4}{3}q^3+\frac{20}{9}q^2 
-\frac{8}{3}q+\frac{11}{9}$\\ \hline
5 & $q^5-2q^4+2q^3-2q^2+2q-1$ & $\frac{5}{9}q^5-\frac{11}{9}q^4+\frac{17}{9}q^3 
-\frac{23}{9}q^2+\frac{23}{9}q-\frac{11}{9}$\\ \hline
6 & $q^6-2q^5+2q^4-2q^3+2q^2-2q+1$ & 
$\frac{40}{81}q^6-\frac{10}{9}q^5+\frac{44}{27}q^4-\frac{22}{9}q^3+\frac{67}{27}q^2-\frac{22}{9}
q+\frac{113}{81}$\\ \hline
%\end{tabular}
%\vspace{0.2 cm}
\caption{Table for $c(n,q,3)_{\rs}$ when $q\equiv2 \imod 3$.}
\label{table2}
\end{longtable}
\end{example}
%%%%%%%%%%%%%%
\subsection{Generating function for $M^{th}$ power semisimple and all conjugacy classes} 

\begin{theorem}\label{gen_func_semisimple_class_revisited}
Let $M\geq 2$ be prime and $(M,q)=1$. Let $t$ be the order of $q$ in $\Z/M\Z^{\times}$. Then,
$$1+\sum\limits_{n=1}^{\infty}c(n,q,M)_{\s}u^n=c(q,u)c(q^t,u^t)^{\frac{1-M}{Mt}}\prod\limits_{k=1}^{
\infty} c(q^t,u^{M^kt})^{\frac{(1-M)^2}{M^{k+1}t}}.$$
\end{theorem}
\begin{proof}
From  Theorem~\ref{gen-fun-primepower} and Lemma~\ref{gen-fun-nhatqd}, we have
\begin{eqnarray*}
&&\displaystyle 1+\sum_{n=1}^{\infty}c(n,q,M)_{ss}u^n =
\left(\frac{1-u^M}{1-qu^M}\right)\prod_{d\geq 1}\left(\frac{1-u^{Md}}{1-u^d} \right)^{N_M(q,d)} \\ 
&=& c(q,u^M).\frac{c(q,u)\prod\limits_{k=0}^{\infty} 
c(q^t,u^{M^kt})^{\frac{1-M}{M^{k+1}t}}}{c(q,u^M)\prod\limits_{k=0}^{\infty} 
c(q^t,u^{M^{k+1}t})^{\frac{1-M}{M^{k+1}t}}}
= c(q,u)c(q^t,u^t)^{\frac{1-M}{Mt}}.\frac{\prod\limits_{k=1}^{\infty} 
c(q^t,u^{M^kt})^{\frac{1-M}{M^{k+1}t}}}{\prod\limits_{k=1}^{\infty} 
c(q^t,u^{M^kt})^{\frac{1-M}{M^kt}}} 
\\&=& c(q,u)c(q^t,u^t)^{\frac{1-M}{Mt}}\prod\limits_{k=1}^{\infty} 
c(q^t,u^{M^kt})^{\frac{(1-M)^2}{M^{k+1}t }}.
\end{eqnarray*}
\end{proof}
\begin{example}
Let us take $M=3$ and $q\equiv 2 \imod 3$. Then, $t=2$ and by the above theorem, 
\begin{eqnarray*}
1+\sum_{n=1}^{\infty}c(n,q,3)_{\s}u^n &=& 
c(q,u)c(q^2,u^2)^{-\frac{1}{3}}\prod\limits_{k=1}^{\infty} 
\left(c(q^2,u^{3^kt})\right)^{\frac{2}{3^{k+1}}} \\ &=& 
\left(\frac{1-u}{1-qu}\right)\left(\frac{1-u^2}{1-q^2u^2}\right)^{-\frac{1}{3}}\left(\frac{1-u^6}{ 
1-q^2u^6}\right)^{\frac{2}{9}}\prod\limits_{k=2}^{\infty} c(q^t,u^{3^kt})^{\frac{2}{3^{k+1}}}
\end{eqnarray*}
\noindent We make a table to compute these classes for some small value of $n$. Recall 
$c(n)_{ss}=q^n-q^{n-1}$.
\begin{table}[ht]
\renewcommand{\arraystretch}{1.2}
\centering
\begin{tabular}{|c|c|c|}
\hline 
$n$ & Number of semisimple classes & Number of $3^{rd}$ power semisimple classes\\ \hline
1 & $q-1$ & $q-1$\\ \hline
2 & $q^2-q$ & $\frac{2}{3}q^2-q+\frac{1}{3}$\\ \hline
3 & $q^3-q^2$ & $\frac{2}{3}q^3-\frac{2}{3}q^2+\frac{1}{3}q-\frac{1}{3}$\\ \hline
4 & $q^4-q^3$ & $\frac{5}{9}q^4-\frac{2}{3}q^3+\frac{2}{9}q^2 
-\frac{1}{3}q+\frac{2}{9}$\\ \hline
5 & $q^5-q^4$ & $\frac{5}{9}q^5-\frac{5}{9}q^4+\frac{2}{9}q^3 
-\frac{2}{9}q^2+\frac{2}{9}q-\frac{2}{9}$\\ \hline
6 & $q^6-q^5$ & 
$\frac{40}{81}q^6-\frac{5}{9}q^5+\frac{5}{27}q^4-\frac{2}{9}q^3+\frac{10}{27}q^2-\frac{2}{9}q-\frac{
4}{81}$\\ \hline
\end{tabular}
\vspace{0.4cm}
\caption{\label{table3} Table for $c(n,q,3)_{\s}$ for $q\equiv 2 \imod 3$.}
\end{table}
\end{example}

Finally, we have the generating function for the $M^{th}$ power conjugacy classes. 

\begin{theorem}\label{gen_func_all_class_revisited}
Let $M\geq 2$ be prime and $(M,q)=1$. Let $t$ be the order of $q$ in $\Z/M\Z^{\times}$. Then,
$$1+\sum\limits_{n=1}^{\infty}c(n,q,M)u^n=\prod_{j=1}^{\infty}\left[c(q,u^j)c(q^t,u^{tj})^{\frac{
1-M }{Mt}}\prod\limits_{k=1}^{\infty} c(q^t,u^{M^ktj})^{\frac{(1-M)^2}{M^{k+1}t}}\right].$$
\end{theorem}
\begin{proof}
The proof is along the same lines as the proof of 
Theorem~\ref{gen_func_semisimple_class_revisited} using 
Corollary~\ref{gen-function-all-classes-prime}.
\end{proof}
We once again work-out an example here.
\begin{example}
Let us take $M=3$ and $q\equiv 2 \imod 3$. Then, $t=2$ and we have,
\begin{eqnarray*}
1+\sum\limits_{n=1}^{\infty}c(n,q,3)u^n  &=& 
\left(\frac{1-u}{1-qu}\right)\left(\frac{1-u^2}{1-qu^2}\right)\left(\frac{1-u^2}{1-q^2u^2}\right)^{
-\frac{1}{3}}\left(\frac{1-u^3}{1-qu^3}\right)\left(\frac{1-u^4}{1-qu^4}\right)\times\\ && 
\left(\frac{1-u^4}{1-q^2u^4}\right)^{-\frac{1}{3}}\left(\frac{1-u^5}{1-qu^5}\right)\left(\frac{1-u^6
}{1-qu^6}\right)\left(\frac{1-u^6}{1-q^2u^6}\right)^{-\frac{1}{9}}P(q,u)
\end{eqnarray*}
where $P(q,u)=1+o(u^7)$. We make a table for $c(n,q,3)$ for small $n$, and compare it with the 
values of $c(n,q)$ given in \cite{Ma}. 

\begin{longtable}{|c|c|c|}
%\renewcommand{\arraystretch}{1.2}
			%\centering
			%\begin{tabular}{|c|c|c|}
\hline 
$n$ & Number of conjugacy classes & Number of $3^{rd}$ power conjugacy classes\\ \hline
1 & $q-1$ & $q-1$\\ \hline
2 & $q^2-1$ & $\frac{2}{3}q^2-\frac{2}{3}$\\ \hline
3 & $q^3-q$ & $\frac{2}{3}q^3+\frac{1}{3}q^2-\frac{2}{3}q-\frac{1}{3}$\\ \hline
4 & $q^4-q$ & $\frac{5}{9}q^4+\frac{2}{9}q^2-q+\frac{2}{9}$\\ \hline
5 & $q^5-q^2-q+1$ & $\frac{5}{9}q^5+\frac{1}{9}q^4+\frac{2}{9}q^3 
-\frac{5}{9}q^2+\frac{7}{9}q+\frac{4}{9}$\\ \hline
6 & $q^6-q^2$ & 
$\frac{40}{81}q^6+\frac{2}{27}q^4+\frac{1}{3}q^3-\frac{11}{27}q^2-\frac{1}{3}q-\frac{13}{81}
$\\ \hline
%\end{tabular}
%\vspace{0.2 cm}
\caption{Table for $c(n,q,3)$ when $q\equiv 2 \imod 3$.}
\label{table4}
\end{longtable}
\end{example}

%%%%%%%%%%%%%%%%%%%%%%%%%%%%%%
\section{Exact value of $M^{th}$ powers in some cases}

In this section, we give the exact formula for $M^{th}$ powers, $M$ is a prime with $(M,q)=1$, 
in $\GL(n,q)$ when $n < Mt$, where $t=\mathfrak{M}(M,q)$ is the order of $q$ in $\Z/M\Z^{\times}$. 
In~\cite{KKS}, the asymptotic value of this proportion as $q\rightarrow \infty$ in finite reductive 
groups is determined. We will show that this limit is achieved in certain cases.

The following is the generating function for the proportion of $M^{th}$ powers in $\GL(n,q)$.
 
\begin{theorem}\label{genfunc-allelements-prime}
Let $M\geq 2$ be a prime and $(M,q)=1$. Then the generating function, 
$ \displaystyle 1+\sum_{n=0}^{\infty}\frac{|\GL(n,q)^M|}{|\GL(n,q)|}u^n = P_2P_1 $
where  
$$P_1=\prod_{d\geq 1}\left(1+\sum\limits_{n\geq 1}\sum_{\lambda\vdash 
n}\frac{u^{Mnd}}{ q^{M^2d\sum_{i} (\lambda^{'}_{i})^2}\smashoperator[r]{\prod_{i\geq 1}} 
\Big(\frac{1}{q^{d}}\Big)_{m_i(\lambda)}} \right)^{N(q,d)-N_M(q,d)},$$

$$P_2= \prod_{d\geq 1}\left( 1+ \sum_{j \geq 1} \sum_{\lambda \vdash j} 
\frac{u^{jd}}{\displaystyle q^{d\sum_{i} (\lambda^{'}_{i})^2} \prod_{i \geq 1} 
\left(\frac{1}{q^d}\right)_{m_i(\lambda)}} \right)^{N_M(q,d)}.$$
\end{theorem} 
\begin{proof}
By Proposition~\ref{prime-case}, we know that $\alpha \in \GL(n,q)$ is $M^{th}$ power if and 
only if for each $f\in \Delta_{\alpha}$, either $f$ is M-power or $M\mid m_j(\lambda_f)$ for 
all $j\geq 1$. Thus, in the cycle index generating function (Equation~\ref{cycle-index}), we 
put, for each $f\in \Phi$, $\lambda$ a partition, 
\[x_{f,\lambda}=
\begin{cases}
1 & \text{; if } f \in \Phi^M\\
1 & \text{; if } f \in \Phi\setminus \Phi^M \text{ and, } M\mid m_j(\lambda) \text{ for 
all } j\geq 1\\
0 & \text{; otherwise}.
\end{cases}
\]
This gives the required formula.
\end{proof}
\begin{lemma}\label{genfunction-allelements}
With the notation as earlier,
$$\displaystyle \prod\limits_{d\geq 1}\left(1+\sum\limits_{n\geq 1}\sum_{\lambda \vdash 
n}\frac{u^{nd}}{ q^{d\sum_{i} (\lambda^{'}_{i})^2}\smashoperator[r]{\prod_{i\geq 1}} 
\Big(\frac{1}{q^d}\Big)_{m_i(\lambda)}} \right)^{N(q,d)}=\frac{1}{1-u}.$$
\end{lemma}
\begin{proof}
The left hand side can be obtained by putting $x_{f,\lambda}=1$ for all $f\in \Phi$ and $\lambda$ 
in the cycle index generating function for $\GL(n,q)$ (see Equation~\ref{cycle-index}). Therefore, 
the coefficient of $u^n$ in the formal power series written in the left hand side is 
$\frac{1}{|\GL(n,q)|}\sum_{\alpha \in \GL(n,q)}1=1$. Thus,
$$\prod\limits_{d\geq 1}\left(1+\sum\limits_{n\geq 1}\sum_{\lambda \vdash n}\frac{u^{nd}}{ 
q^{d\sum_{i} (\lambda^{'}_{i})^2}\smashoperator[r]{\prod_{i\geq 1}} 
\Big(\frac{1}{q^d}\Big)_{m_i(\lambda)}} \right)^{N(q,d)}=1+u+u^2+\cdots=\frac{1}{1-u}.$$
\end{proof}

From~\cite{KKS} we recall the following. The limit points of $\frac{|\GL(n,q)^M|}{|\GL(n,q)|}$ 
considered as a sequence in $q$ with $M$ and $n$ fixed (see Section 4, \cite{KKS}), except 
possibly the value$1$, is  given as follows: For each $t\in \mathbb{N}$ such that $t\mid 
M-1$,
$$P(n,t,M)=\sum_{\substack{\lambda \vdash n \\  \lambda=1^{m_1}2^{m_2}\cdots}} 
\frac{1}{M^{\pi_t(\lambda)}\displaystyle\prod_{i\geq 1}i^{m_i}m_i!}$$
where $\pi_t(\lambda)$ denotes the number of parts of $\lambda$ divisible by $t$. The 
generating function for $P(n,t,M)$ in the argument $n$ (see Proposition 4.6, \cite{KKS}) is given 
as 
follows:
\begin{equation}\label{genfunction_bound}
\displaystyle 1+\sum_{n=1}^{\infty}P(n,t,M)u^n = \frac{(1-u^t)^{\frac{M-1}{Mt}}}{(1-u)}.
\end{equation}
We have the following,
\begin{theorem}\label{formula_M-powers_GL}
Let $M$ be a prime and $(M,q)=1$. Let $t$ be the order of $q$ in $\mathbb{Z}/M\mathbb{Z}^{\times}$. 
Then,
\begin{equation*}
\displaystyle \frac{|\GL(n,q)^M|}{|\GL(n,q)|}= \sum_{\substack{\lambda \vdash n \\  
\lambda=1^{m_1}2^{m_2}\cdots}} \frac{1}{M^{\pi_t(\lambda)}\displaystyle\prod_{i\geq 1}i^{m_i}m_i!},
\end{equation*}
when $n<Mt$.
\end{theorem}
\begin{proof}
From Theorem~\ref{genfunc-allelements-prime}, we have,
$$1+\sum_{n=0}^{\infty}\frac{|\GL(n,q)^M|}{|\GL(n,q)|}u^n =P_2 P_1.$$
Let us denote $\widehat{N}(q,d)=N(q,d)-N_M(q,d)$. Now, we analyze $P_2$ using 
Lemma~\ref{genfunction-allelements}. We have,
$$ P_2 = \frac{1}{1-u} \prod_{d\geq 1}\left( 1+ \sum_{j \geq 1} \sum_{\lambda \vdash j} 
\frac{u^{jd}}{\displaystyle q^{d\sum_{i} (\lambda^{'}_{i})^2} \prod_{i \geq 1} 
\left(\frac{1}{q^d}\right)_{m_i(\lambda)}} \right)^{-\widehat{N}(q,d)} = \frac{P_3}{1-u}
$$
where,
\begin{eqnarray*}
P_3 &=& \prod_{d\geq 1}\left( 1+ \sum_{j \geq 1} \sum_{\lambda \vdash j} 
\frac{u^{jd}}{\displaystyle q^{d\sum_{i} (\lambda^{'}_{i})^2} \prod_{i \geq 1} 
\left(\frac{1}{q^d}\right)_{m_i(\lambda)}} \right)^{-\widehat{N}(q,d)} \\ &=& \prod_{t\mid d}\left( 
1+ \sum_{j \geq 1} \sum_{\lambda \vdash j} \frac{u^{jd}}{\displaystyle q^{d\sum_{i} 
(\lambda^{'}_{i})^2} \prod_{i \geq 1} \left(\frac{1}{q^d}\right)_{m_i(\lambda)}} 
\right)^{-\widehat{N}(q,d)} \\ &=& \prod_{k=0}^{\infty} \prod_{M\nmid d}\left( 1+ \sum_{j \geq 1} 
\sum_{\lambda \vdash j} \frac{u^{jM^ktd}}{\displaystyle q^{td\sum_{i} (\lambda^{'}_{i})^2} \prod_{i 
\geq 1} \left(\frac{1}{q^{td}}\right)_{m_i(\lambda)}} \right)^{-\widehat{N}(q,M^k.t.d)} \\ &=& 
\prod_{k=0}^{\infty} \prod_{M\nmid d}\left( 1+ \sum_{j \geq 1} \sum_{\lambda \vdash j} 
\frac{u^{jM^ktd}}{\displaystyle q^{td\sum_{i} (\lambda^{'}_{i})^2} \prod_{i \geq 1} 
\left(\frac{1}{q^{td}}\right)_{m_i(\lambda)}} \right)^{\frac{1-M}{M^{k+1}.t}N(q^{M^kt},d)}\\
&=& \prod_{M\nmid d}\left( 1+ \sum_{j \geq 1} \sum_{\lambda \vdash j} 
\frac{u^{tjd}}{\displaystyle q^{td\sum_{i} (\lambda^{'}_{i})^2} \prod_{i \geq 1} 
\left(\frac{1}{q^{td}}\right)_{m_i(\lambda)}} \right)^{N(q^t,d)} \times \\ && \prod_{k=1}^{\infty} 
\prod_{M\nmid d}\left( 1+ \sum_{j \geq 1} \sum_{\lambda \vdash j} \frac{u^{jM^ktd}}{\displaystyle 
q^{td\sum_{i} (\lambda^{'}_{i})^2} \prod_{i \geq 1} \left(\frac{1}{q^{td}}\right)_{m_i(\lambda)}} 
\right)^{\left(N(q^t,M^k.d)+\frac{N(q^{M^{k-1}t},d)}{M^k}\right)\frac{1-M}{Mt}}\\
&=& \prod_{d\geq 1}\left( 1+ \sum_{j \geq 1} \sum_{\lambda \vdash j} 
\frac{u^{tjd}}{\displaystyle q^{td\sum_{i} (\lambda^{'}_{i})^2} \prod_{i \geq 1} 
\left(\frac{1}{q^{td}}\right)_{m_i(\lambda)}} \right)^{N(q^t,d).\frac{1-M}{Mt}} \times \\ && 
\prod_{k=1}^{\infty} \prod_{M\nmid d}\left( 1+ \sum_{j \geq 1} \sum_{\lambda \vdash j} 
\frac{u^{jM^ktd}}{\displaystyle q^{td\sum_{i} (\lambda^{'}_{i})^2} \prod_{i \geq 1} 
\left(\frac{1}{q^{td}}\right)_{m_i(\lambda)}} \right)^{N(q^{M^{k-1}t},d)\frac{1-M}{M^{k+1}t}} 
\end{eqnarray*}

\begin{eqnarray*}
&=& \left(\frac{1}{1-u^t}\right)^{\frac{1-M}{Mt}}\prod_{k=1}^{\infty} \prod_{M\nmid d}\left( 
1+ \sum_{j \geq 1} \sum_{\lambda \vdash j} \frac{u^{jM^ktd}}{\displaystyle q^{td\sum_{i} 
(\lambda^{'}_{i})^2} \prod_{i \geq 1} \left(\frac{1}{q^{td}}\right)_{m_i(\lambda)}} 
\right)^{\frac{N(q^{M^{k-1}t},d)}{M^{k-1}}\frac{1-M}{M^2t}}.
\end{eqnarray*}
Note that we have used Lemma~\ref{recursive-reln} in the last few steps in the same way as in 
the proof of Proposition~\ref{gen-fun-nhatqd}.
Inductively, we get,
$$P_3=\prod_{k=0}^{\infty}(1-u^{M^kt})^{\frac{M-1}{M^{k+1}t}}.$$
Thus, 
\begin{eqnarray*}
\displaystyle 1+\sum_{n=0}^{\infty}\frac{|\GL(n,q)^M|}{|\GL(n,q)|}u^n &=& 
P_2P_1=\frac{P_3}{1-u}P_1=\frac{\prod\limits_{k=0}^{\infty}(1-u^{M^kt})^{\frac{M-1}{M^{k+1}t}}}{1-u}
P_1 \\ &=& 
P_1\frac{(1-u^t)^{\frac{M-1}{Mt}}}{1-u}\prod\limits_{k=1}^{\infty}(1-u^{M^kt})^{\frac{M-1}{M^{k+1}t}
}.
\end{eqnarray*}
Finally observe that when $n<Mt$, the coefficient of the generating function on the right-hand side 
is contributed only by the coefficient of $\frac{(1-u^t)^{\frac{M-1}{Mt}}}{1-u}$ which is precisely 
$P(n,t,M)$ (see Equation~\ref{genfunction_bound}).
This proves the required result.
\end{proof}
\noindent Using this we can simplify the generating function in 
Theorem~\ref{genfunc-allelements-prime} further.
\begin{corollary}\label{genfunction-allelements-prime-simple}
$$\displaystyle 
1+\sum_{n=0}^{\infty}\frac{|\GL(n,q)^M|}{|\GL(n,q)|}u^n=\frac{P_1}{1-u}\prod\limits_{k=0}^{\infty}
(1-u^{M^kt})^{\frac{M-1}{M^{k+1}t}},$$
where $\displaystyle P_1=\prod_{d\geq 1}\left(1+\sum\limits_{n\geq 1}\sum_{\lambda\vdash 
n}\frac{u^{Mnd}}{ q^{M^2d\sum_{i} (\lambda^{'}_{i})^2}\smashoperator[r]{\prod_{i\geq 1}} 
\Big(\frac{1}{q^{d}}\Big)_{m_i(\lambda)}} \right)^{N(q,d)-N_M(q,d)}.$
\end{corollary}
The power map on $\GL(n,q)$ is surjective if and only if $(M,q)=1$ and $n<t$ where $t$ is the 
order of $q$ in $\Z/M\Z^{\times}$ (see Proposition 4.2, \cite{KKS}). This is reflected in the above 
proposition since $P(n,t,M)=1$ if and only if $n<t$. The quantity $P(n,t,M)$ gives the 
sub-sequential limits as a sequence in $q$ of the proportion of $M^{th}$ powers in $\GL(n,q)$ for 
a fixed $n$. The previous theorem shows that in some cases these sub-sequential limits are actually 
achieved. 

%%%%%%%%%%%%%%%%%%%%%%%%%%%%%%%%%%%%%%%%
\section{$M^{th}$ powers when $M$ is a prime and $q$ is a power of $M$}

So far, we have dealt with the case when $M$ is coprime to $q$. Recall that we are interested in 
determining the image of the power map $\omega \colon \text{GL}(n, q) \rightarrow \text{GL}(n, q)$ 
given by $x\mapsto x^M$. From the point of view of Jordan decomposition of elements, when $(q,M)=1$, 
all unipotent elements survive as they are of order a power of $q$. Now, we want to focus on the 
case when $M$ and $q$ are not coprime. For simplicity of computations, we take the case $M$ is a 
prime and $q$ is a power of $M$. In this case, all semisimple elements survive in the image. Miller 
(see~\cite{Mi}) enumerated squares in $\text{GL}(n, 2^a)$; thus dealt with a particular case, $M=2$, 
of our situation. Our exposition in this section follows closely to that of Miller. However, more 
often than not, we need to write the proof of various statements a fresh, thus generalising his 
results. 

In this section, we fix $M$ a prime and $q$, a power of $M$. We determine the conjugacy classes that 
are $M^{th}$ powers in $\text{GL}(n, q)$. We begin with a Lemma (analogous to our earlier 
Lemma~\ref{jordan-block-power}) which is~\cite[Lemma 2]{Mi}. Recall the notation that $J_{\gamma, 
n}$ denotes a Jordan block matrix of size $n$ with diagonal $\gamma$.  
\begin{lemma}[Miller]\label{jordan-block-power-modular}
Let $J_{0,n}$ be the Jordan block corresponding to scalar $0$. Then, $J_{0,n}^M$ is conjugate 
to 
$$\underbrace{J_{0, \ceil*{\frac{n}{M}}} \oplus\cdots \oplus J_{0, \ceil*{\frac{n}{M} }}}_{\bar 
n} \oplus \underbrace{ J_{0, \floor*{\frac{n}{M}}}\oplus \cdots \oplus J_{0, 
\floor*{\frac{n}{M}}}}_{(M- \bar n)}$$
		where $0\leq \bar n \leq M-1$ is $n\imod M$ and $\ceil*{\frac{n}{M}}$, $\floor*{\frac{n}{M}}$ are the ceiling and floor functions respectively. 
	\end{lemma}
	
	%%%%%%%%%%%%%%%%%%%%%%%%%%
	\subsection{A map on partitions}
In the view of Lemma above, we define the following map $\Theta_M$ on the set of partitions 
$\Lambda$. Let $\Lambda(n)$ denote the set of all partitions of $n$; thus $\Lambda=\bigcup_{n\geq 0} 
\Lambda(n)$. We define the map $\Theta_M\colon \Lambda(n)\rightarrow \Lambda(n)$ as follows: Given 
a partition $\lambda=(\lambda_1, \ldots ,\lambda_k)$ of $n$, we define 
$$\Theta_M(\lambda) = \left(\underbrace{\ceil*{\frac{\lambda_1}{M}},\ldots, 
\ceil*{\frac{\lambda_1}{M}}}_{\bar \lambda_1}, \underbrace{\floor*{\frac{\lambda_1}{M}},\ldots, 
\floor*{\frac{\lambda_1}{M}}}_{M -\bar \lambda_1 },   \ldots,  
\underbrace{\ceil*{\frac{\lambda_k}{M}},\ldots, \ceil*{\frac{\lambda_k}{M}}}_{\bar \lambda_k}, 
\underbrace{\floor*{\frac{\lambda_k}{M}},\ldots, \floor*{\frac{\lambda_k}{M}}}_{M -\bar \lambda_k}  
\right)$$
suitably rearranged in non-increasing order, where $0\leq \bar \lambda_i \leq (M-1)$ is $\lambda_i 
\imod M$. We give some examples to illustrate this 
map.
\begin{center}
	\begin{tabular}{|c|c|| c | c|}
	\hline 
			$\lambda \vdash 4$ & $\Theta_2(\lambda)$ & $\lambda \vdash 4$ & $\Theta_2(\lambda)$  \\  
\hline 
			$(3,1)$  & $(2,1,1)$ & $(2, 2)$ & $(1, 1, 1, 1)$ \\ \hline
			$(1, 1, 1, 1)$ & $(1, 1, 1, 1)$ & $(4) $ & $(2, 2)$ \\ \hline
			$(2, 1, 1)$ & $(1, 1, 1, 1)$ && \\ \hline
		\end{tabular}
	\end{center}
	Thus, we see that $\Theta_2(\Lambda(4))=\{(1, 1, 1, 1), (2, 1, 1), (2, 2) \}\subset \Lambda(4)$ 
and similarly $\Theta_3(\Lambda(5))= \{(1, 1, 1, 1, 1), (2, 1, 1, 1), (2, 2, 1)\}\subset 
\Lambda(5)$. Since, this map is quite 
important for us, we elaborate this alternatively using the other notation for a partition. Let 
$\lambda=1^{m_1(\lambda)}2^{m_2(\lambda)}\cdots$ be a partition of $n$ then $\Theta_M(\lambda) = 
1^{m_1(\Theta_M(\lambda))} 2^{m_2(\Theta_M(\lambda))} \cdots$ where 
$$m_i(\Theta_M(\lambda)) = Mm_{(iM)}(\lambda) + \sum_{j=1}^{M-1} (M-j) \left( m_{(iM-j)}(\lambda)+ 
m_{(iM+j)}(\lambda)\right).$$ 
If we take $M=2$, we get the function defined in~\cite[Proposition 3]{Mi}.   It is 
easy to see that both of the above definitions are same. 
	We make a table to illustrate the size of image for some small values.
	\begin{center}
		\begin{tabular}{|c|c|c|c|c|} 
			\hline
			$n$ & $|\Lambda(n)$ & $|\Theta_2(\Lambda(n))|$ & $|\Theta_3(\Lambda(n))|$ & $|\Theta_5(\Lambda(n))|$  \\ 
			\hline
			$1$ & $1$& $1$ & $1$ & $1$ \\
			$2$ & $2$& $1$ & $1$ & $1$\\ 
			$3$ & $3$ & $2$ & $1$ & $1$\\ 
			$4$ &$5$ & $3$ & $2$ & $1$\\ 
			$5$ & $7$& $4$ & $3$ & $1$\\ 
			$6$ & $11$& $5$ & $4$ & $2$\\ 
			$7$ & $15$& $7$ & $5$ & $3$\\ 
			$8$ & $22$& $10$ & $6$ & $4$\\
			$9$ & $30$& $13$ & $7$ & $5$\\ 
			$10$ & $42$&  $16$ & $9$ & $6$\\ 
			$11$ & $56$& $21$ & $12$ & $7$\\ 
			$12$ & $77$& $28$ & $16$ & $8$\\ 
			$13$ & $101$& $35$ & $20$ & $9$\\ 
			$14$ & $135$& $43$ & $24$ & $10$\\ 
			$15$ & $176$& $55$ & $28$ & $11$\\ 
			\hline
		\end{tabular}
	\end{center}
	\label{Tab:T}
	\noindent We would like to count the image of $\Theta_M$. The following Lemma is a generalization of~\cite[Proposition 3]{Mi}. 
	\begin{lemma}
		Let $\Theta_M \colon \Lambda(n)\rightarrow \Lambda(n)$ be the map described above. Then, a 
partition $\mu$ of $n$ is in the image of $\Theta_M$ if and only if $\displaystyle\sum_{j=1}^{M-1} 
m_{(iM-j)}(\mu')\leq 1$, for each $i\geq 1$, where $\mu'$ is transpose of the partition $\mu$.
	\end{lemma}
	\begin{proof}
		The proof is along the same lines as in~\cite{Mi}.
	\end{proof}
	\noindent Now we can write the generating function for this quantity. This generalises the result mentioned in~\cite{Mi} (just before Proposition 3) for $M=2$.
	\begin{proposition}\label{image-thetaM}
		With the notation as above,
		$$ 1 + \sum_{n=1}^{\infty} |\Theta_M(\Lambda(n))| u^n = \prod_{k=1}^{\infty} \frac{1 + u^{kM-1} + u^{kM-2} + \cdots + u^{kM-(M-1)}}{1 - u^{kM}}.$$
	\end{proposition}
	\begin{proof}
		By the previous Lemma, we see that the $k^{th}$ term $|\Theta_M(\Lambda(k))|$ is equal to the number of partitions $\lambda\vdash k$ satisfying $\sum_{j=1}^{M-1} m_{(iM-j)}(\lambda')\leq 1$. This means that for each $i\geq 1$, at most one of $m_{iM-j}(\lambda')=1$, and all other terms are $0$ in the sum. For counting sake, we can think of $\lambda$ instead of $\lambda'$. Thus, $|\Theta_M(\Lambda(k))|$ is the coefficient if $u^{k}$ in the following product:
		\begin{eqnarray*}
			&&\left ((1+u+u^2+\cdots + u^{M-1})\left(\sum_{t=0}^{\infty} u^{tM} \right)\right)\times \\ && \left((1+u^{M+1}+u^{M+2}+\cdots + u^{2M-1})\left(\sum_{t=0}^{\infty}u^{2tM}\right)\right)\times \cdots \\
			&& \cdots \times \left((1+ u^{iM+1} + u^{iM+2} + \cdots + u^{iM+(M-1)}) \left(\sum_{t=0}^{\infty}u^{itM}\right)\right)\times \cdots \\  
			&& = \prod_{i=1}^{\infty} \left( 1 + u^{(i-1)M+1} + u^{(i-1)M+2} + \cdots + u^{(i-1)M+(M-1)}\right)\left( \frac{1}{1-u^{iM}}\right)\\
			&& = \prod_{i=1}^{\infty}  \frac{ 1 + u^{iM-1} + u^{iM-2} + \cdots + u^{iM-(M-1)}}{1-u^{iM}}.
		\end{eqnarray*}
		This completes the proof.
	\end{proof}
	
	%%%%%%%%%%%%%%%%%%%%%%%%%%%%%%%%%
	
	\subsection{Computing powers}
	Now, we are ready to describe the result which generalises ~\cite[Theorem 1]{Mi}, and determines $M^{th}$ powers in $\text{GL}(n,q)$ in this case.
	\begin{theorem}\label{M-power-modular-case}
		Let $M$ be a prime and $q$ be a power of $M$. Let $\alpha \in \text{GL}(n, q)$ and $\Delta_{\alpha}$ be its associated combinatorial data consisting of $f_i$ and $\lambda_{f_i}$. Then, $X^M=\alpha$ has a solution in $\text{GL}(n,q)$ if and only if the partitions $\lambda_{f_i}$ are in $\Theta_M\left(\Lambda(|\lambda_{f_i}|)\right)$, for all $i$. 
	\end{theorem}
	\begin{proof}
		Let $A\in \text{GL}(n, q)$ be a solution of $X^M=\alpha$. It suffices to prove the statement when $A$ has a single Jordan block. Thus we may assume, $A$ corresponds to the polynomial $g$ and partition $\mu_g=(\mu_1, \ldots, \mu_k)$. If $g(x)=(x - a_1)\cdots (x - a_d)$ then define $g^{(M)}(x)= (x-a_1^M)\cdots (x-a_d^M)$. Clearly, $g^{(M)}$ is defined over $\mathbb F_q$ if $g$ is so. Now, we claim that the associated combinatorial data to $A^M$ is $g^{(M)}$ and the partition $\Theta_M(\mu)$. Since, raising power $M$ is a bijection on $\mathbb F_q^*$, we can easily find $g$ such that $g^{(M)} = f$ (for example, by factorising it as a product of linear polynomials). Thus, this gives the required condition that $\lambda_f$ must be $\Theta_M(\mu)$.      
		
		For the converse, we have $\alpha$ with its combinatorial data satisfying $\lambda_{f_i}\in \Theta_M\left(\Lambda(|\lambda_{f_i}|)\right)$, for all $i$. Without loss of generality, we may assume it has a single Jordan block, say $\alpha$ is conjugate to $J_{f, k}$. Rest of the proof is similar to the~\cite[Corollary 3 and Corollary 4]{Mi}, thus we mention it briefly. By factorising $f$ over $\bar{\mathbb F_q}$, we can reduce it to constructing the solution $A$ for the Jordan matrix $J_{\beta, m}$ where $\beta$ is a root of $f$.  We take $A=J_{\gamma, m}$ and get $A^M= \gamma^M I+ J_{0,m}^M$ (since $q$ is an $M$ power). By Lemma~\ref{jordan-block-power-modular}, the combinatorial data $\Delta_{A^M}$ consists of polynomial $(x- \gamma^M)$ and the partition $\mu=\left( \underbrace{\ceil*{\frac{m}{M}},\ldots,\ceil*{\frac{m}{M}}}_{\bar m}, \underbrace{\floor*{\frac{m}{M}},\ldots,\floor*{\frac{m}{M}}}_{M-\bar m}\right)$. Thus, we choose $\gamma$ so that $\gamma^M=\beta$. Combined with the fact that $\mu \in \Theta_M(\Lambda(m))$, and putting together the Galois conjugate blocks, we get the proof.
	\end{proof} 
	\noindent  Since,  order of a semisimple element $\alpha$ is coprime to $M$, the equation $X^M=\alpha$ always has a solution in $\text{GL}(n,q)$. Further, 
	\begin{corollary}
		With notation as above, let $\alpha \in \text{GL}(n,q)$ be a regular element. Then, $X^M=\alpha$ has a solution in $\text{GL}(n,q)$ if and only if $\alpha$ is semisimple.
	\end{corollary}
	\begin{proof}
		Since $\alpha$ is regular,  the combinatorial data $\Delta_{\alpha}$ consists of $f_i$ and $\lambda_{f_i}$ with exactly one part $|\lambda_{f_i}|$. Then by Theorem~\ref{M-power-modular-case}, $X^M=\alpha$ has a solution if and only if, for each $i$, the partition $\lambda_{f} = (|\lambda_{f_i}|)$ is in $\Theta_M(|\lambda_{f_i}|)$. Now, from definition of $\Theta_M$, this is possible only if $|\lambda_{f_i}|=1$ for all $i$. This proves that $X^M=\alpha$ has a solution if and only if $\alpha$ is semisimple.
	\end{proof}
	\noindent We summarise this as follows:
	\begin{proposition}
		Let $M$ be a prime and $q$ be a power of $M$. Then,
		\begin{enumerate}
			\item the $M^{th}$ power semisimple classes in $\text{GL}(n,q)$ are $c(n,q,M)_{\s}= c(n,q)_{\s}$. The generating function for semisimple classes (respectively semisimple elements) which are $M^{th}$ power is same as that of all semisimple classes (respectively semisimple elements). 
			\item The $M^{th}$ power regular and regular semisimple classes in $\text{GL}(n,q)$ are $c(n,q,M)_{\rg} = c(n,q,M)_{\rs} = c(n,q)_{\rs}$. The generating function for regular and regular semisimple classes (respectively elements) which are $M^{th}$ power is same as that of all regular semisimple classes (respectively elements). 
		\end{enumerate}
	\end{proposition}
	
	The following result generalizes~\cite[Theorem 2]{Mi}.
	\begin{theorem}\label{Theorem-modular}
		Let $M$ be a prime and $q$ be a power of $M$. The generating function for $M^{th}$ power conjugacy classes in $\text{GL}(n,q)$ is,
		$$1+\sum_{n=1}^{\infty}c(n,q,M) u^n = \prod_{d\geq 1}\left(\prod_{k\geq 1} \frac{1 + u^{d(kM-1)} + \cdots + u^{d(kM-(M-1))}}{1-u^{dkM}}\right)^{N(q,d)}.$$
	\end{theorem}
	\begin{proof}
		By Theorem~\ref{M-power-modular-case} we have (the first equality below), 
		\begin{eqnarray*}
			&& 1+\sum_{n=1}^{\infty} c(n,q,M) u^n  = \sum_{\lambda_f \in \Theta_M(|\lambda_f|)} u^{\sum_{f\in \phi}|\lambda_{f}|. deg(f)}=  \prod_{f\in \Phi} \sum_{\lambda_f \in \Theta_M(|\lambda_f|)} u^{|\lambda_{f}|. deg(f)}\\
			&=& \prod_{f\in \Phi} \prod_{k\geq 1} \frac{1+u^{deg(f).(kM-1)} + ux^{deg(f).(kM-2)} + \cdots + u^{deg(f).(kM-(M-1))}}{1-u^{deg(f).kM}} \\
			&=& \prod_{d\geq 1}\left(\prod_{k\geq 1} \frac{1+ u^{d(kM-1)} + \cdots + u^{d(kM-(M-1))}}{1-u^{dkM}}\right)^{N(q,d)}.
		\end{eqnarray*}
		The third equality follows from Proposition~\ref{image-thetaM} by taking $u$ as $u^{deg(f)}$.
	\end{proof}
	\noindent We note that for $M=2$, we get~\cite[Theorem 2]{Mi} by substituting $q=2$ in the following.
	\begin{corollary}
		For $M=2$ we have,
		$$1+\sum_{n=1}^{\infty} c(n,q,2) u^n = \prod_{n\geq 1} \frac{(1-u^{2n})(1-qu^{2n})}{(1+u^{2n-1})(1-qu^{n})(1-qu^{4n})}.$$
	\end{corollary}
	\begin{proof}
		From previous Theorem we have, 
		$$1+\sum_{n=1}^{\infty} c(n,2) u^n = \prod_{d\geq 1}\prod_{k\geq 1} \left(\frac{1+u^{d(2k-1)}}{1-u^{2dk}}\right)^{N(q,d)}.$$
		Since $\displaystyle \prod_{k\geq 1} \frac{1+u^{2k-1}}{1-u^{2k}} = \prod_{k\geq 1}\frac{1-u^{2k}}{(1-u^k)(1-u^{4k})}$ (obtained by multiplying with $(1-u^{2k-1})$ in the numerator and denominator) we get
		$$1 + \sum_{n=1}^{\infty} c(n,2) u^n = \prod_{d\geq 1}\prod_{k\geq 1} \left(\frac{1-u^{2dk}}{(1-u^{dk})(1-u^{4dk})}\right)^{N(q,d)}.$$
		Now, we use $\prod_{d\geq 1}(1- u^d)^{-N(q,d)}=\frac{1-y}{1-qy}$ to get the required result.
	\end{proof}
	We can obtain the formula for general elements as follows:
	\begin{proposition}
		With the notation as above, 
		{\small $$\displaystyle 1+\sum_{n=1}^{\infty}\frac{|\text{GL}(n,q)^M|}{|\text{GL}(n,q)|} u^n = \prod_{f\in \Phi} \left( 1+ \sum_{n\geq 1} \sum_{\substack{\lambda \vdash n\\ \lambda \in \Theta_M(\Lambda(n))}} \frac{u^{n.deg(f)}}{ q^{deg(f).\sum_{i} (\lambda^{'}_{i})^2} \smashoperator[r]{\prod_{i\geq 1}} \left(\frac{1}{q^{deg(f)}}\right)_{m_i(\lambda_f)}} \right).$$}
	\end{proposition}
	\begin{proof}
		This follows from Theorem~\ref{M-power-modular-case} and the cycle index generating function Equation~\ref{cycle-index}.
	\end{proof}
	
	\bibliographystyle{amsalpha}
	\bibliography{references}

\end{document}